\documentclass[10pt,a4paper,reqno,sumlimits]{amsart}

\usepackage[numbers]{natbib}
\usepackage{amsmath}
\usepackage{amssymb}
\usepackage[mathscr]{euscript}
\usepackage{enumerate}
\usepackage{xspace}
\usepackage{color}
\usepackage{enumerate}
\usepackage{enumitem}
\usepackage{amsthm}
\usepackage[
margin=3cm,
includefoot,
footskip=30pt,
]{geometry}
\begin{document}
\renewcommand{\mathscr}{\mathcal}
\theoremstyle{plain}
\newtheorem{C}{Convention}
\newtheorem{fact}{Fact}
\newtheorem*{SA}{Standing Assumption}
\newtheorem{theorem}{Theorem}
\newtheorem{condition}{Conditions}[section]
\newtheorem{lemma}{Lemma}
\newtheorem{proposition}{Proposition}
\newtheorem{corollary}{Corollary}
\newtheorem{claim}[theorem]{Claim}
\newtheorem{definition}{Definition}
\newtheorem{Ass}[theorem]{Assumption}
\newcommand{\q}{Q}
\theoremstyle{definition}
\newtheorem{remark}{Remark}
\newtheorem{note}[theorem]{Note}
\newtheorem{example}{Example}
\newtheorem{assumption}[theorem]{Assumption}
\newtheorem*{notation}{Notation}
\newtheorem*{assuL}{Assumption ($\mathbb{L}$)}
\newtheorem*{assuAC}{Assumption ($\mathbb{AC}$)}
\newtheorem*{assuEM}{Assumption ($\mathbb{EM}$)}
\newtheorem*{assuES}{Assumption ($\mathbb{ES}$)}
\newtheorem*{assuM}{Assumption ($\mathbb{M}$)}
\newtheorem*{assuMM}{Assumption ($\mathbb{M}'$)}
\newtheorem*{assuL1}{Assumption ($\mathbb{L}1$)}
\newtheorem*{assuL2}{Assumption ($\mathbb{L}2$)}
\newtheorem*{assuL3}{Assumption ($\mathbb{L}3$)}
\newtheorem{charact}[theorem]{Characterization}
\newcommand{\notiz}{\textup} 
\renewenvironment{proof}{{\parindent 0pt \it{ Proof:}}}{\mbox{}\hfill\mbox{$\Box\hspace{-0.5mm}$}\vskip 16pt}
\newenvironment{proofthm}[1]{{\parindent 0pt \it Proof of Theorem #1:}}{\mbox{}\hfill\mbox{$\Box\hspace{-0.5mm}$}\vskip 16pt}
\newenvironment{prooflemma}[1]{{\parindent 0pt \it Proof of Lemma #1:}}{\mbox{}\hfill\mbox{$\Box\hspace{-0.5mm}$}\vskip 16pt}
\newenvironment{proofcor}[1]{{\parindent 0pt \it Proof of Corollary #1:}}{\mbox{}\hfill\mbox{$\Box\hspace{-0.5mm}$}\vskip 16pt}
\newenvironment{proofprop}[1]{{\parindent 0pt \it Proof of Proposition #1:}}{\mbox{}\hfill\mbox{$\Box\hspace{-0.5mm}$}\vskip 16pt}
\newcommand{\s}{\u s}

\newcommand{\Law}{\ensuremath{\mathop{\mathrm{Law}}}}
\newcommand{\loc}{{\mathrm{loc}}}
\newcommand{\Log}{\ensuremath{\mathop{\mathscr{L}\mathrm{og}}}}
\newcommand{\Meixner}{\ensuremath{\mathop{\mathrm{Meixner}}}}
\newcommand{\of}{[\hspace{-0.06cm}[}
\newcommand{\gs}{]\hspace{-0.06cm}]}
\newcommand{\mlambda}{\lambda_\textup{max}}
\let\MID\mid
\renewcommand{\mid}{|}
\newcommand{\bn}{\|_\mathbb{B}}
\newcommand{\kn}{\|_\mathbb{K}}
\newcommand{\bs}{\rangle_\mathbb{B}}
\newcommand{\ks}{\rangle_\mathbb{K}}
\newcommand{\dn}{\|_{\mathbb{R}^n}}

\let\SETMINUS\setminus
\renewcommand{\setminus}{\backslash}

\def\stackrelboth#1#2#3{\mathrel{\mathop{#2}\limits^{#1}_{#3}}}

\renewcommand{\theequation}{\thesection.\arabic{equation}}
\numberwithin{equation}{section}

\newcommand\llambda{{\mathchoice
      {\lambda\mkern-4.5mu{\raisebox{.4ex}{\scriptsize$\backslash$}}}
      {\lambda\mkern-4.83mu{\raisebox{.4ex}{\scriptsize$\backslash$}}}
      {\lambda\mkern-4.5mu{\raisebox{.2ex}{\footnotesize$\scriptscriptstyle\backslash$}}}
      {\lambda\mkern-5.0mu{\raisebox{.2ex}{\tiny$\scriptscriptstyle\backslash$}}}}}

\newcommand{\prozess}[1][L]{{\ensuremath{#1=(#1_t)_{0\le t\le T}}}\xspace}
\newcommand{\prazess}[1][L]{{\ensuremath{#1=(#1_t)_{0\le t\le T^*}}}\xspace}

\newcommand{\tr}{\operatorname{tr}}
\newcommand{\lijepoa}{{\mathscr{A}}}
\newcommand{\lijepob}{{\mathscr{B}}}
\newcommand{\lijepoc}{{\mathscr{C}}}
\newcommand{\lijepod}{{\mathscr{D}}}
\newcommand{\lijepoe}{{\mathscr{E}}}
\newcommand{\lijepof}{{\mathscr{F}}}
\newcommand{\lijepog}{{\mathscr{G}}}
\newcommand{\lijepok}{{\mathscr{K}}}
\newcommand{\lijepoo}{{\mathscr{O}}}
\newcommand{\lijepop}{{\mathscr{P}}}
\newcommand{\lijepoh}{{\mathscr{H}}}
\newcommand{\lijepom}{{\mathscr{M}}}
\newcommand{\lijepou}{{\mathscr{U}}}
\newcommand{\lijepov}{{\mathscr{V}}}
\newcommand{\lijepoy}{{\mathscr{Y}}}
\newcommand{\cF}{{\mathscr{F}}}
\newcommand{\cG}{{\mathscr{G}}}
\newcommand{\cH}{{\mathscr{H}}}
\newcommand{\cM}{{\mathscr{M}}}
\newcommand{\cD}{{\mathscr{D}}}
\newcommand{\bD}{{\mathbb{D}}}
\newcommand{\bF}{{\mathbb{F}}}
\newcommand{\bG}{{\mathbb{G}}}
\newcommand{\bH}{{\mathbb{H}}}
\newcommand{\dd}{d} 
\newcommand{\ddd}{\operatorname{d}}
\newcommand{\er}{{\mathbb{R}}}
\newcommand{\ce}{{\mathbb{C}}}
\newcommand{\erd}{{\mathbb{R}^{d}}}
\newcommand{\en}{{\mathbb{N}}}
\newcommand{\de}{{\mathrm{d}}}
\newcommand{\im}{{\mathrm{i}}}
\newcommand{\indik}{{\mathbf{1}}}
\newcommand{\D}{{\mathbb{D}}}
\newcommand{\E}{E}
\newcommand{\N}{{\mathbb{N}}}
\newcommand{\Q}{{\mathbb{Q}}}
\renewcommand{\P}{{\mathbb{P}}}
\newcommand{\ud}{\operatorname{d}\!}
\newcommand{\ii}{\operatorname{i}\kern -0.8pt}
\newcommand{\cadlag}{c\`adl\`ag }
\newcommand{\p}{P}
\newcommand{\F}{\mathbf{F}}
\newcommand{\1}{\mathbf{1}}
\newcommand{\f}{\mathscr{F}^{\hspace{0.03cm}0}}
\newcommand{\lle}{\langle\hspace{-0.085cm}\langle}
\newcommand{\rre}{\rangle\hspace{-0.085cm}\rangle}
\newcommand{\llbr}{[\hspace{-0.085cm}[}
\newcommand{\rrbr}{]\hspace{-0.085cm}]}

\def\EM{\ensuremath{(\mathbb{EM})}\xspace}

\newcommand{\la}{\langle}
\newcommand{\ra}{\rangle}

\newcommand{\Norml}[1]{%
{|}\kern-.25ex{|}\kern-.25ex{|}#1{|}\kern-.25ex{|}\kern-.25ex{|}}

\title[Limit Theorems for Martingale Problems]{Limit Theorems for Cylindrical Martingale Problems \\associated with L\'evy Generators} 
\author[D. Criens]{David Criens}
\address{D. Criens - Technical University of Munich, Center for Mathematics, Germany}
\email{david.criens@tum.de}

\keywords{Cylindrical Martingale Problem, L\'evy Generator, Limit Theorem, Feller Process, Stochastic Partial Differential Equation, Jump-Diffusion Existence Theorem\vspace{1ex}}


\thanks{D. Criens - Technical University of Munich, Center for Mathematics, Germany,  \texttt{david.criens@tum.de}.}


\date{\today}
\maketitle

\frenchspacing
\pagestyle{myheadings}

\begin{abstract}
We prove limit theorems for cylindrical martingale problems associated to L\'evy generators. Furthermore, we give sufficient and necessary conditions for the Feller property of well-posed problems with continuous coefficients. We discuss two applications. First, we derive continuity and linear growth conditions for the existence of weak solutions to infinite-dimensional stochastic differential equations driven by L\'evy noise. Second, we derive continuity, local boundedness and linear growth conditions for limit theorems and the Feller property of weak solutions to stochastic partial differential equations driven by Wiener noise.
\end{abstract}

\section{Introduction}
Cylindrical martingale problems (MPs) associated with L\'evy generators can be considered as the martingale formulation of (analytically and probabilistically) weak and mild solutions to (semilinear) stochastic (partial) differential equations (S(P)DEs) driven by L\'evy noise. 
As in the classical finite-dimensional case, the martingale formulation gives access to weak conditions for the strong Markov property and Girsanov-type theorems, see \cite{criens18,EJP2924}. 
Another application of MPs, which was impressively exploited by Stroock and Varadhan \cite{SV} and Jacod and Shiryaev \cite{JS} in the finite-dimensional case, are limit theorems. 
In this article, we show the following: A sequence of solutions to MPs whose initial laws converge weakly and whose coefficients converge uniformly on compact sets to continuous coefficients converge weakly to a solution of the MP associated with the limiting initial law and the limiting coefficients. 
Moreover, we prove that under a uniqueness and existence assumption on the limiting MP a localized tightness condition implies weak convergence. This observation can be used to verify tightness under boundedness or moment conditions.

Let us mention two consequences of our observations.
Following Stroock and Varadhan \cite{SV}, we call a family of solutions to a well-posed MP a Feller family (or simply Feller) if it is weakly continuous w.r.t. their initial values. 
The limit theorem shows that a family of well-posed MPs with continuous coefficients is Feller if and only if a tightness condition holds. This observation generalizes results known for finite- and infinite-dimensional cases, see \cite{EJP2924,Stroock75,SV}. 
The limit theorem can also be used to construct solutions to MPs from solutions of approximate MPs and hence provides existence results in the spirit of Skorokhod's theorem for finite-dimensional SDEs. 

For illustration we discuss two applications: First, we derive an existence theorem for weak solutions to infinite-dimensional SDEs of the type
\begin{align}\label{int: SDE}
\dd Y_t = b(Y_{t-})\dd t + \sigma (Y_{t-})\dd W_t + \int v(x, Y_{t-})(p - q)(\dd x, \dd t),
\end{align}
where \(W\) is a cylindrical Brownian motion and \(p  - q\) is a compensated Poisson random measure. 
To explain our result in more detail, let \(J\) be a positive symmetric compact operator on a separable Hilbert space \(\mathbb{B}\). Under the assumption that \(b, \sigma\) and \(v\) satisfy a continuity and local boundedness assumption on \(\mathbb{B}\) and a linear growth condition on \(J(\mathbb{B})\), we show that the SDE \eqref{int: SDE} has a weak solution with values in \(\mathbb{B}\). 

For our second application we consider MPs corresponding to stochastic evolution equations of the form
\begin{align}\label{SPDE}
\dd Y_t = \left(A Y_t + b(Y_t)\right)\dd t + \sigma(Y_t)\dd W_t, 
\end{align} where \(A\) is the generator of a compact \(C_0\)-semigroup. We adapt the compactness method from \cite{doi:10.1080/17442509408833868} to show that a localized version of tightness holds if the non-linearities satisfy local boundedness conditions. 
It follows from this observation that solution families to well-posed diffusion-type MPs with continuous and locally bounded non-linearities are Feller. To the best of our knowledge, this result is new. In addition, we derive limit theorems either under a well-posedness and a local boundedness condition or a uniqueness and a linear growth condition. Our results apply, for instance, to non-linear stochastic heat equations.

Next, we comment on related literature in infinite-dimensional frameworks. 
For settings allowing jumps we are only aware of limit theorems for semimartingales given in \cite{XIE1995277}. 
We also think that the literature contains no existence result for infinite-dimensional SDEs with jumps which is comparable to ours. For continuous noise an existence theorem similar to ours was given in \cite{GMR09}. We stress that the setting in \cite{GMR09} is not necessarily Markovian. Our result strengthens the Markovian version of the theorem from \cite{GMR09} by replacing one of the linear growth conditions with a local boundedness condition.
A limit theorem for diffusions under a tightness and local boundedness condition can be found in \cite{EJP2924}. Our result extends the theorem from \cite{EJP2924} by showing that the tightness is implied by a well-posedness assumption. 
Under well-posedness, continuity, linear growth and moment conditions on the initial laws, limit theorems for time-inhomogeneous SPDEs driven by Wiener noise were proven in \cite{doi:10.1080/07362999708809484}. For the time-homogeneous case we extend these results by replacing the linear growth conditions with local boundedness conditions and confirm the conjecture from \cite{doi:10.1080/07362999708809484} that no moment assumption on the initial laws is needed. 
The Feller property of S(P)DEs with Wiener noise is under frequent investigation. We mention two related papers: \cite{doi:10.1080/17442508.2012.712973, Maslowski1999}.
In \cite{doi:10.1080/17442508.2012.712973} the martingale formulation is used to identify the transition semigroup of the studied Cauchy problem, while the core argument is based on a perturbation result for semigroups. The approach in \cite{Maslowski1999} is based on Girsanov's theorem. 

The article is structured as follows: In Section \ref{2.1} we introduce the MP, following the exposition given in \cite{criens18}. In Section \ref{sec: MR} we state our main results, in Section \ref{sec:sm} we discuss the existence of weak solutions to SDEs of the type \eqref{int: SDE} and in Section \ref{sec:3} we discuss the diffusion case. The proofs are collected in Section \ref{sec: Proofs}.

\section{Cylindrical Martingale Problems} \label{2.1}
Let \((\mathbb{B}, \|\cdot\|)\) be a real separable reflexive Banach space, which we equip with its norm topology, and let \(\mathbb{B}^*\) be its (topological) dual, which we equip with the operator norm \(\|\cdot\|_o\) and the corresponding topology. It is well-known that \((\mathbb{B}^*, \|\cdot\|_o)\) is also a real separable reflexive Banach space. 
For \(x^*\in \mathbb{B}^*\) and \(x \in \mathbb{B}\) we write
\[
x^*(x) \triangleq \langle x, x^*\rangle.
\]

We define \(\Omega\) to be the Skorokhod space of all \cadlag functions \(\alpha \colon \mathbb{R}_+\to \mathbb{B}\) and equip it with the Skorokhod topology, which turns it into a Polish space, see \cite{EK,JS}. 
The coordinate process \(X\) on \(\Omega\) is defined by \(X_t(\alpha) = \alpha(t)\) for all \(\alpha \in \Omega\) and \(t \geq 0\). Moreover, we set \[\mathscr{F} \triangleq \sigma(X_s, s \geq 0), \quad\mathscr{F}_t \triangleq \sigma(X_s, s \in [0, t]), \quad \F\triangleq (\mathscr{F}_t)_{t \geq 0}.\]
It is well-known that \(\mathscr{F}\) is the Borel \(\sigma\)-field on \(\Omega\).
Except otherwise stated, all terms such as stopping time, martingale, local martingale etc. refer to \(\F\) as the underlying filtration.

For a stopping time \(\xi\) we write
\[
\mathscr{F}_{\xi} \triangleq \big\{ A \in \mathscr{F} \colon A \cap \{\xi \leq t\} \in \mathscr{F}_t \textup{ for all } t \geq 0\big\}, 
\]
which is easily seen to be a \(\sigma\)-field. 

An operator \(Q \colon \mathbb{B}^* \to \mathbb{B}\) is called positive if \(\la Q x^*, x^*\ra \geq 0\) for all \(x^* \in \mathbb{B}^*\) and symmetric if \(\la Q x^*, y^*\ra = \la Q y^*, x^*\ra\) for all \(x^*, y^* \in \mathbb{B}^*\).         
We denote by \(S^+(\mathbb{B}^*, \mathbb{B})\) the set of all linear, bounded, positive and symmetric operators \(\mathbb{B}^* \to \mathbb{B}\).     
For a topological space \(E\) we denote the corresponding Borel \(\sigma\)-field by \(\mathscr{B}(E)\).

Next, we introduce the parameters for the martingale problem:
\begin{enumerate}
	\item[(i)]
	Let \(A \colon D(A) \subset \mathbb{B} \to \mathbb{B}\) be a linear, densely defined and closed operator.
	Here, \(D(A)\) denotes the domain of the operator \(A\).
	\item[(ii)]
	Let \(b \colon \mathbb{B} \to\mathbb{B}\) be Borel and such that for all bounded sequences \((y^*_n)_{n \in \mathbb{N}} \subset \mathbb{B}^*\) and all bounded sets \(G \subset \mathbb{B}\) it holds that 
	\begin{align*}
	\sup_{n \in \mathbb{N}} \sup_{x \in G} |\la b(x), y^*_n\ra| < \infty.
	\end{align*}
	\item[(iii)] 
	Let \(a\colon \mathbb{B}  \to S^+(\mathbb{B}^*, \mathbb{B})\) be bounded on bounded subsets of \(\mathbb{B}\) and Borel, i.e. \(x \mapsto a(x) y^*\) is Borel for all \(y^* \in \mathbb{B}^*\). Here, bounded refers to the operator norm.
	\item[(iv)]
	Let \(K\) be a Borel transition kernel from \(\mathbb{B}\) into \(\mathbb{B}\), such that for all bounded sequences \((y^*_n)_{n \in \mathbb{N}} \subset \mathbb{B}^*\), all bounded sets \(G \subset \mathbb{B}\) and all \(\epsilon > 0\) it holds that
	\begin{align*} 
	\sup_{n \in \mathbb{N}} &\sup_{x \in G} \int_{\mathbb{B}} \1_{\{\|y\| \leq \epsilon\}} |\la y, y^*_n\ra|^2K(x, \dd y) < \infty,\end{align*}\begin{align*}
	\sup_{x \in G} K(x, \{z \in \mathbb{B} \colon \|z\| \geq \epsilon\}) < \infty,
	\end{align*}
	and \(K(\cdot, \{0\}) = 0\).
	\item[(v)]
	Let \(\eta\) be a probability measure on \((\mathbb{B}, \mathscr{B}(\mathbb{B}))\). 
\end{enumerate}

Let \(A^*\) be the Banach adjoint of \(A\)
and let \(C^2_{c}(\mathbb{R}^d)\) be the set of twice continuously differentiable functions \(\mathbb{R}^d\to \mathbb{R}\) with compact support. 
The set of test functions for the MP is given by the following:
\[\mathcal{C} \triangleq \big\{g(\langle \cdot, y^*_1\rangle, \dots, \langle \cdot, y^*_n\rangle) \colon g \in C^2_c(\mathbb{R}^n), y^*_1, \dots, y^*_n \in D(A^*), n \in \mathbb{N}\big\}.\]
If \(g\) is twice continuously differentiable and \(f = g(\langle \cdot, y^*_1\rangle, \dots, \langle \cdot, y^*_n\rangle)\), we write \(\partial_{i} f\) for the partial derivative \[(\partial_{i} g)(\la \cdot, y^*_1\ra, \dots, \la \cdot, y^*_n\ra)\] and \(\partial^2_{ij} f\) for the partial derivative
\[
(\partial^2_{ij} g)(\la \cdot, y^*_1\ra, \dots, \la \cdot, y^*_n\ra).
\]

A bounded Borel function \(h\colon \mathbb{B} \to \mathbb{B}\) is called truncation function if there exists an \(\epsilon > 0\) such that \(h(x) = x\) for all \(x \in \mathbb{B}\) with \(\|x\| \leq \epsilon\). Throughout the article we fix a truncation function \(h\).

For \(f = g(\langle \cdot, y^*_1\rangle, \dots, \langle \cdot, y^*_n\rangle) \in \mathcal{C}\) we set
\begin{equation} \label{K}
\begin{split}
\mathcal{K}f (x) \triangleq  \sum_{i = 1}^n (& \langle x, A^* y^*_i\rangle + \langle b(x), y^*_i \rangle)\partial_i f(x) +  \frac{1}{2} \sum_{i=1}^n\sum_{j = 1}^n \langle a(x)y^*_i, y^*_j\rangle \partial^2_{ij}f(x) 
\\&+ \int_{\mathbb{B}} \Big( f(x + y) - f(x) - \sum_{i = 1}^n \langle h(y), y^*_i\rangle \partial_i f(x) \Big) K(x, \dd y).
\end{split}
\end{equation}
We are in the position to define the martingale problem.
\begin{definition} \label{def: MP}
	We call a probability measure \(P\) on \((\Omega, \mathscr{F})\) a \emph{solution to the martingale problem (MP) \((A, b,a,K, \eta)\)}, if the following hold:
	\begin{enumerate}
		\item[\textup{(i)}]
		\(P \circ X^{-1}_0 = \eta\). 
		\item[\textup{(ii)}]
		For all \(f \in \mathcal{C}\) the process
		\begin{equation}\label{f - K}
		\begin{split}
		M^f \triangleq f(X) - f(X_0) - \int_0^{\cdot} \mathcal{K}f (X_{s-})\dd s
		\end{split}
		\end{equation}
		is a local \(P\)-martingale. 
	\end{enumerate}
The set of solutions is denoted by \(\mathcal{M}(A, b, a, K, \eta)\).
	We say that the MP is \emph{well-posed}, if there exists a unique solution for all degenerated initial laws, i.e. for all \(\eta = \varepsilon_x\), \(x \in \mathbb{B}\), where \(\varepsilon_x\) is the Dirac measure on \(x \in \mathbb{B}\).
\end{definition}
For reader's convenience, we have collected some important results for MPs in Appendix \ref{app: 1}.
In particular, due to Proposition \ref{prop: A2} in Appendix \ref{app: 1}, the set
 \begin{align*} 
 \mathcal{D} \triangleq \left\{g(\la \cdot, y^*\ra)\colon g \in C^2_c(\mathbb{R}), y^* \in D(A^*)\right\} \subset \mathcal{C}
 \end{align*}
 determines the solutions of a MP and, due to Proposition \ref{prop: A1} in Appendix \ref{app: 1}, if \(P \in \mathcal{M}(A, b,a, K,  \eta)\), then  \(M^f\) is a local \(P\)-martingale for all 
 \begin{align}\label{eq: B}
 f \in \mathcal{B} \triangleq \big\{ g (\langle \cdot, y^*\rangle) \colon g \in C^2_b(\mathbb{R}), y^* \in D(A^*) \big\}, 
 \end{align}
 where \(C^2_b(\mathbb{R})\) denotes the set of bounded twice continuously differentiable functions with bounded gradient and bounded Hessian matrix.

\section{Limit Theorems for Cylindrical Martingale Problems}\label{sec: MR}
In this section we state our main results. We start with the limit theorem as described in the introduction.
Let us stress that we implicitly assume that all coefficients to MPs satisfy the assumptions introduced in the previous section.
We formulate three conditions:
\begin{enumerate}
		\item[\textup{(M1)}]
	The map
	\(
	x \mapsto \mathcal{K} f(x)
	\) 
	is continuous for all \(f \in \mathcal{D}\). Here, \(\mathcal{K}\) is defined as in \eqref{K}.
	\item[\textup{(M2)}] 
	Define \(\mathcal{B}^n\) as in \eqref{eq: B} with \(A^*\) replaced by \((A^n)^*\) and define \(\mathcal{K}^n\) as in \eqref{K} with \((A, b, a, K)\) replaced by \((A^n, b^n, a^n, K^n)\). For all \(f \in \mathcal{D}\) there exists a sequence \((f^n)_{n \in \mathbb{N}}\) with \(f^n \in \mathcal{B}^n\) such that for all \(m \in \mathbb{N}\) \begin{align}\label{eq: bdd assp} \sup_{n \in \mathbb{N}} \sup_{x \in \mathbb{B}} \big| f^n(x)\big| + \sup_{n \in \mathbb{N}} \sup_{\|x\| \leq m} \big| \mathcal{K}^n f^n (x)\big| < \infty,\end{align}
	and 
	\begin{align*} \big|f^n - f\big| + 
	\big|\mathcal{K}^n f^n - \mathcal{K} f\big| \to 0
	\end{align*}
	as \(n \to \infty\) uniformly on compact subsets of \(\mathbb{B}\). 
	\item[\textup{(M3)}] \(\eta^n \to \eta\) weakly as \(n \to \infty\). 
\end{enumerate}
We are in the position to state the first main result of this article.
\begin{theorem}\label{theo1}
	For all \(n \in \mathbb{N}\) let \(P^n\) be a solution to the MP \((A^n, b^n, a^n, K^n, \eta^n)\) and assume that \textup{(M1), (M2)} and \textup{(M3)} hold.
If \(P\) is a probability measure on \((\Omega, \mathscr{F})\) such that \(P^n \to P\) weakly as \(n \to \infty\), then \(P\) solves the MP \((A, b, a, K, \eta)\).
\end{theorem}
The proof is given in Section \ref{proof: theo1} below. For diffusions a related result is given by \cite[Lemma 4.3]{EJP2924}.
As a first corollary we give an existence result under a tightness condition and a convergence result under a uniqueness assumption.
\begin{corollary}\label{coro: existence}
	For all \(n \in \mathbb{N}\) let \(P^n\) be a solution to the MP \((A^n, b^n, a^n, K^n, \eta^n)\) and assume that \textup{(M1), (M2)} and \textup{(M3)} hold. 
	\begin{enumerate}
		\item[\textup{(i)}] If \((P^n)_{n \in\mathbb{N}}\) is tight, then the MP \((A, b, a, K, \eta)\) has a solution.
				\item[\textup{(ii)}] If \((P^n)_{n \in \mathbb{N}}\)  is tight and the MP \((A, b, a, K, \eta)\) has at most one solution, then MP \((A, b, a, K, \eta)\) has a unique solution \(P\) and \(P^n \to P\) weakly as \(n \to \infty\).
	\end{enumerate}
\end{corollary}
\begin{proof}
	(i). Because \((P^n)_{n \in \mathbb{N}}\) is tight, we can extract a weakly convergent subsequence due to Prohorov's theorem. Due to Theorem \ref{theo1}, the limit point of this subsequence solves the MP \((A, b, a, K, \eta)\). 
	
	(ii). Due to Theorem \ref{theo1}, any accumulation point of \((P^n)_{n \in \mathbb{N}}\) is a solution to the MP \((A, b, a, K, \eta)\). Thus, because this MP is assumed to have at most one solution, we conclude that all accumulation points coincide with the unique solution \(P\) and \(P^n \to P\) weakly as \(n \to \infty\) follows from \cite[Theorem 2.6]{bil99}.
\end{proof}
As a second corollary, we obtain a characterization of the Feller property of MPs. It can be viewed as a generalization of \cite[Corollary 4.4]{EJP2924} to a setup including jumps. 
\begin{corollary}\label{coro: Feller}
	Assume that \textup{(M1)} holds and that for all \(x \in \mathbb{B}\) the MP \((A, b, a, K, \varepsilon_x)\) has a unique solution \(P_x\).
	Then, the following are equivalent: 
	\begin{enumerate}
		\item[\textup{(i)}]
		The family \(\{P_x, x \in \mathbb{B}\}\) is Feller, i.e. \(x \mapsto P_x\) is weakly continuous, which means that \(P_{x_n} \to P_x\) weakly as \(n \to \infty\) whenever \(x_n \to x\) as \(n \to \infty\).
		\item[\textup{(ii)}] For all sequences \((x_n)_{n \in \mathbb{N}} \subset \mathbb{B}\) with \(x_n \to x \in \mathbb{B}\) as \(n \to \infty\), the sequence \((P_{x_n})_{n \in \mathbb{N}}\) is tight.
	\end{enumerate}
\end{corollary}
\begin{proof}
	The implication (i) \(\Rightarrow\) (ii) is due to Prohorov's theorem and the implication (ii) \(\Rightarrow\) (i) follows from Corollary \ref{coro: existence} and Proposition \ref{prop: A1} in Appendix \ref{app: 1}.
\end{proof}
\begin{remark}
	If \(x \mapsto P_x\) is Borel, the following are equivalent:
	\begin{enumerate}
		\item[(i)]
		The family \(\{P_x, x \in \mathbb{B}\}\) is Feller.
		\item[(ii)]
		The map \(\eta \mapsto \int P_x \eta(\dd x)\) is weakly continuous, i.e. \(\int P_x \eta^n(\dd x) \to \int P_x \eta(\dd x)\) weakly as \(n \to \infty\) whenever \(\eta^n \to \eta\) weakly as \(n \to \infty\).
	\end{enumerate}
	This fact is noteworthy, because in case for all \(x \in \mathbb{B}\) the MP \((A, b, a, K, \varepsilon_x)\) has a unique solution \(P_x\), the map \(x \mapsto P_x\) is Borel and \(\int P_x \eta(\dd x)\) is the unique solution to the MP \((A, b, a, K,\eta)\), see Proposition \ref{prop: A3} in Appendix \ref{app: 1}.
\end{remark}

Many conditions for tightness include boundedness or moment conditions. In such cases, it might be easier to consider a localized version of \((P^n)_{n \in \mathbb{N}}\). We introduce the stopping time \begin{align}\label{eq: tau}
\tau_z (\alpha) \triangleq \inf(t \geq 0\colon \|\alpha(t-)\| \geq z \textup{ or } \|\alpha(t)\| \geq z),\quad z \geq 0,\ \alpha \in \Omega,
\end{align}
see \cite[Proposition 2.1.5]{EK}.
As the following theorem shows, if a good candidate for the limit of \((P^n)_{n \in \mathbb{N}}\) exists, it suffices to show tightness for the localized sequences
\((P^n \circ X^{-1}_{\cdot \wedge \tau_m})_{n \in \mathbb{N}}\) and all \(m \in \mathbb{N}\). 
In Section \ref{sec:3} below we use this observation together with the compactness method from \cite{doi:10.1080/17442509408833868} to deduce mild conditions for the Feller property of diffusion-type MPs and limit theorems.
\begin{theorem}\label{prop: loc}
	For all \(n \in \mathbb{N}\) let \(P^n\) be a solution to the MP \((A^n, b^n, a^n, K^n, \eta^n)\) and assume that \textup{(M1), (M2)} and \textup{(M3)} hold.  If for all \(x \in \mathbb{B}\) the MP \((A, b, a, K, \varepsilon_x)\) has a unique solution \(P_x\) and for all \(m \in \mathbb{N}\) the sequence \((P^n \circ X_{\cdot \wedge \tau_m}^{-1})_{n \in \mathbb{N}}\) is tight, then the map \(x  \mapsto P_x\) is Borel, \(\int P_x \eta(\dd x)\) is the unique solution to the MP \((A, b, a, K, \eta)\) and \(P^n \to \int P_x \eta(\dd x)\) weakly as \(n \to \infty\).
\end{theorem}
The proof is given in Section \ref{proof: theo2} below. In general, tightness of stochastic processes is well-studied. Sufficient and necessary conditions in various settings can be found in \cite{EK, 10.2307/1427238, metivier1987, XIE1995277}. A frequently used condition is the following version of Aldous's tightness criterion.
\begin{proposition}\label{prop: gen tight}
	Let \((P^n)_{n \in \mathbb{N}}\) be a sequence of probability measures on \((\Omega, \mathcal{F})\) such that the following hold: \begin{enumerate}
		\item[\textup{(i)}] For all \(t \geq 0\) the sequence \((P^n \circ X^{-1}_t)_{n \in \mathbb{N}}\) is tight.
		\item[\textup{(ii)}] For all \(\epsilon > 0, M \in \mathbb{N}\) and all sequences \((\rho_n, h_n)_{n \in \mathbb{N}}\), where \((\rho_n)_{n \in \mathbb{N}}\) is a sequence of stopping times such that \(\sup_{n \in \mathbb{N}} \rho_n \leq M\) and \((h_n)_{n \in \mathbb{N}} \in (0, \infty)\) is a sequence such that \(h_n \to 0\) as \(n \to \infty\), we have 
		\[
		\lim_{n \to \infty} P^n \Big(\|X_{\rho_n + h_n} - X_{\rho_n}\| \geq \epsilon \Big) = 0.
		\]
	\end{enumerate}
Then, \((P^n)_{n \in \mathbb{N}}\) is tight. 
\end{proposition}
\begin{proof}
	See \cite[Theorem 6.8]{walshintroduction} and \cite[Corollary p. 120]{10.2307/3212499}.
\end{proof}
We illustrate an application of this proposition and Corollary \ref{coro: existence} in the next section.

\section{Existence of Weak Solutions to Jump-Diffusion SDEs}\label{sec:sm}
In this section we apply our results in a semimartingale setting. Namely, we give continuity and linear growth conditions for the existence of weak solutions to jump-diffusion SDEs of the type
\begin{align}\label{eq: main SDE}
\dd Y_t = b(Y_{t-}) \dd t + \sigma(Y_{t-}) \dd W_t + \int v(x, Y_{t-})(p - q)(\dd x, \dd t),
\end{align}
where \(W\) is a cylindrical standard Brownian motion and \(p - q\) is a compensated Poisson random measure (see \cite[Section II.1]{JS}).

We assume that \(\mathbb{B}\) is a separable Hilbert space. Moreover, let \((E, \mathcal{E})\) be a Blackwell space (see \cite[Section II.1.a]{JS}; any Polish space with its Borel \(\sigma\)-field is a Blackwell space), \(\mathbb{H}\) be a second separable Hilbert spaces and \(q = \dd t \otimes F\) be the compensator of a Poisson random measure on \(\mathbb{R}_+ \times E\) (see \cite[Theorem II.1.8]{JS}).  The norm of  \(\mathbb{B}\) is denoted by \(\|\cdot\bn\) and the corresponding scalar product is denoted by \(\langle \cdot, \cdot \bs\). 

Let \(J \colon \mathbb{B} \to \mathbb{B}\) be a compact operator of the form 
\[
J x = \sum_{i = 1}^\infty \lambda_i \langle x, h_i\bs h_i, \quad x \in  \mathbb{B}, \lambda_k > 0 \text{ for all } k \in \mathbb{N},
\]
where \(\sup_{k \in \mathbb{N}} \lambda_k < \infty\) and \((h_k)_{k \in \mathbb{N}}\) is an orthonormal basis of \(\mathbb{B}\). 
Symmetric positive compact operators are always of this form except that the coefficients \(\lambda_k\) are not necessarily strictly positive, see \cite[Theorem VI.3.2]{werner2007funktionalanalysis}. 
We also define \(J^{-1} \colon J(\mathbb{B}) \to \mathbb{B}\) by 
	\[
	J^{-1} x \triangleq \sum_{i = 1}^\infty \frac{1}{\lambda_i} \langle x, h_i\bs h_i, \quad x \in J(\mathbb{B}).
	\]
Equipped with the scalar product \(\langle \cdot, \cdot\ks \triangleq \langle J^{-1}\ \cdot\ , J^{-1}\ \cdot\ \bs\) the space \(\mathbb{K} \triangleq J(\mathbb{B})\) becomes a separable Hilbert space with orthonormal basis \((e_k)_{k \in \mathbb{N}} \triangleq (\lambda_k h_k)_{k \in \mathbb{N}}\), see \cite[Proposition C.0.3]{roeckner15}. We denote the corresponding norm by \(\|\cdot\kn\).
As always, we identify \(\mathbb{B},\mathbb{H}\) and \(\mathbb{K}\) with their (topological) duals.

\begin{lemma}\label{lem: emb comp}
The embedding \(\iota \colon \mathbb{K} \hookrightarrow \mathbb{B}\) is compact. 
\end{lemma}
\begin{proof} Fix a bounded sequence \((y_n)_{n \in \mathbb{N}} \subset \mathbb{K}\). There exists a sequence \((x_n)_{n \in \mathbb{N}} \subset \mathbb{B}\) such that \(y_n = J x_n\) for \(n \in \mathbb{N}\). Because 
\[
\|y_n\kn= \| J^{-1} J x_n\bn = \|x_n\bn, 
\]
the sequence \((x_n)_{n \in \mathbb{N}}\) is bounded in \(\mathbb{B}\). Consequently, because \(J\) is compact, the sequence \((y_n)_{n \in \mathbb{N}} = (J x_n)_{n \in \mathbb{N}}\) has a convergent subsequence in \(\mathbb{B}\). This shows that \(\iota\) is compact.
\end{proof}

The space of all Hilbert-Schmidt operators \(\mathbb{H} \to \mathbb{K}\) is denote by \(L_2(\mathbb{H}, \mathbb{K})\) and the corresponding Hilbert-Schmidt norms is denoted by 
\(\|\cdot\|_{\textup{HS}(\mathbb{K})}\). As always, we denote the space of linear bounded operators \(\mathbb{H} \to \mathbb{B}\) by \(L(\mathbb{H}, \mathbb{B})\). In case the image and preimage spaces are not \(\mathbb{B}\) and \(\mathbb{H}\) the notation is analogously. The operator norm on \(L(\mathbb{H}, \mathbb{B})\) is denoted by \(\|\cdot\|_{L(\mathbb{B})}\).

\begin{theorem}\label{prop: existence}
Let \(b \colon \mathbb{B} \to \mathbb{B}, \sigma \colon \mathbb{B} \to L(\mathbb{H}, \mathbb{B})\) and \(v \colon E \times \mathbb{B} \to \mathbb{B}\) be measurable maps such that the following hold:
\begin{enumerate}
\item[\textup{(i)}]
\(b (\mathbb{K}) \subset \mathbb{K}, \sigma (\mathbb{K}) \subset L_2(\mathbb{H}, \mathbb{K})\) and \(v(E \times \mathbb{K}) \subset \mathbb{K}\).

\item[\textup{(ii)}]
For all \(y \in E\) the maps 
\[
\mathbb{B} \ni x \mapsto b(x), \sigma (x), v(y, x)
\]
are continuous. For \(\sigma\) we mean continuity in the operator norm \(\|\cdot\|_{L(\mathbb{B})}\).
\item[\textup{(iii)}]
There exists a Borel function \(\gamma \colon E \to \mathbb{R}_+\) such that \(\int_E \gamma^2(y)F(\dd y) < \infty\) and a constant \(L \in (0, \infty)\) such that for all \(y \in E\) and \(x \in \mathbb{K}\)
\begin{align}
\|b(x)\kn + \|\sigma(x)\|_{\textup{HS} (\mathbb{K})} &\leq L\big(1 + \|x\kn\big),\label{eq: lg}
\\
\|v(y, x)\kn &\leq \gamma(y)  \big(1 + \|x\kn\big).\nonumber
\end{align}
Moreover, for all bounded sets \(B \subset \mathbb{B}\) we have
\begin{align}\label{eq: lbdd}
\sup_{x \in B} \|b(x)\bn + \sup_{x \in B} \|\sigma (x)\|_{L(\mathbb{B})} < \infty, 
\end{align}
and there exists a Borel function \(\zeta_B\colon E \to \mathbb{R}_+\) such that \(\int_E \zeta^2_B(z)F(\dd z) < \infty\) and 
\begin{align}\label{eq: loc bdd B}
\sup_{x \in B} \|v(y, x)\bn \leq  \zeta_B (y)
\end{align}
for all \(y \in E\).
\end{enumerate}
For all probability measures \(\eta\) on \((\mathbb{K}, \mathcal{B}(\mathbb{K}))\) there exists a solution to the MP \((0, \mu, a, K, \eta \circ \iota^{-1})\), where for all \(x \in \mathbb{B}\)
\begin{align*}
\mu (x) &\triangleq b(x) + \int_E \big(h(v(y, x)) - v(y, x)\big) F(\dd y),\\
a (x) &\triangleq \sigma (x) \sigma(x)^*,\\
K(x, G) &\triangleq \int_E \1_{G\backslash \{0\}} (v(y, x)) F(\dd y),\quad G \in \mathcal{B}(\mathbb{B}),
\end{align*}
and \(\sigma (x)^* \in L(\mathbb{B}, \mathbb{H})\) denotes the adjoint of \(\sigma (x) \in L(\mathbb{H}, \mathbb{B})\). 
\end{theorem}
We prove this theorem in Section \ref{sec: pf sm} below. Lipschitz conditions for the existence of (pathwise) unique solutions to SDEs of the type \eqref{eq: main SDE} can be found in \cite{doi:10.1080/17442501003624407,metivier}. A version of Theorem \ref{prop: existence} for SDEs driven by Wiener noise is given in \cite{GMR09}. We are not aware of any result in the direction of Theorem \ref{prop: existence} which allows L\'evy noise.

\begin{remark}
	\begin{enumerate}
	\item[\textup{(i)}] In case \(\mathbb{B}\) is finite-dimensional it is possible to take \(J = \textup{Id}\). Then, Theorem \ref{prop: existence} boils down to a classical  Skorokhod-type existence result for jump-diffusion SDEs.
	\item[\textup{(ii)}] Because there exists an \(\epsilon > 0\) such that \(h(x) = x\) for all \(x \in \mathbb{B} \colon \|x\| \leq \epsilon\), there exists a constant \(l > 0\) such that 
	\(
	\|h(x) - x\bn \leq l \|x\bn^2. 
	\)
	Thus, \eqref{eq: lbdd} and \eqref{eq: loc bdd B} imply that \(\mu\) as defined in Theorem \ref{prop: existence} satisfies
	\[
	\sup_{x \in G} \|\mu(x)\bn < \infty, \quad G \subset \mathbb{B} \text{ bounded}.
	\]
\item[(iii)]
	Note that for \(T \in L_2(\mathbb{H}, \mathbb{K})\) we have
	\begin{align*}
	\|T\|_{\textup{HS}(\mathbb{K})} &= \sum_{i = 1}^\infty\sum_{j = 1}^\infty \langle T u_i, e_j\ks^2 \\&= \sum_{i = 1}^\infty\sum_{j = 1}^\infty \langle J^{-1} T u_i, J^{-1} e_j\bs^2
	\\&= \sum_{i = 1}^\infty\sum_{j = 1}^\infty \langle J^{-1} T u_i, h_j\bs^2 \\&= \|J^{-1}  T\|_{\textup{HS}(\mathbb{B})},
	\end{align*}
	where \((u_k)_{k \in \mathbb{N}}\) is an orthonormal basis of \(\mathbb{H}\).
	Thus, \eqref{eq: lg} is implied by the linear growth condition (iii\('\)) as defined in \cite[Remark 2]{GMR09}.
	In Theorem \ref{prop: existence} the second linear growth condition (ii\('\)) from \cite[Remark 2]{GMR09} is weakened to the local boundedness assumption \eqref{eq: lbdd}.
		\end{enumerate}
\end{remark}

Let us comment on the idea of proof. The argument is based on Corollary \ref{coro: existence}, i.e. we construct an approximation sequence and verify its tightness via Aldous's tightness criterion, see Proposition \ref{prop: gen tight}. In infinite-dimensional Hilbert spaces closed balls are not compact and, consequently, moment bounds do not imply tightness. To overcome this problem we construct the approximation sequence on the smaller Hilbert space \(\mathbb{K}\). Because \(\mathbb{K}\) is compactly embedded in \(\mathbb{B}\), moment bounds for the approximation sequence on \(\mathbb{K}\) imply tightness on the larger Hilbert space \(\mathbb{B}\).

\section{A Diffusion Setting}\label{sec:3}
In this section we discuss the diffusion case as an important special case of our setting.
\subsection{The Setup}\label{sec: 5.1}
We slightly adjust our setup. Let \(\Omega\) be the space of all continuous functions \(\mathbb{R}_+ \to \mathbb{B}\), where \(\mathbb{B}\) is assumed to be a separable Hilbert space. As usual, we identify \(\mathbb{B}\) with its (topological) dual and equip \(\Omega\) with the local uniform topology. We set \(X, \mathscr{F}\) and \(\F = (\mathscr{F}_t)_{t \geq 0}\) as in Section \ref{2.1}.
Also in this case, \(\mathscr{F}\) is the Borel \(\sigma\)-field on \(\Omega\). 
We define \(\tau_z\) as in Section \ref{2.1}. Due to the continuous paths of \(X\) we have 
\[
\tau_z = \inf (t \geq 0 \colon \|X_t\| \geq z).
\]
For diffusions the coefficient \(K\) is not relevant and we remove it from all notations. The MP is defined as in Definition \ref{def: MP}. 
	Due to \cite[Problem 25, p. 153]{EK}, Theorems \ref{theo1} and \ref{prop: loc} also hold in this setting.

Let \(\mathbb{H}\) be a second separable Hilbert space and let \(b^n\) and \(a^n\) be as follows:
\begin{enumerate}
	\item[(a)] The coefficients \(b^n\colon \mathbb{B} \to \mathbb{B}\) are Borel.
	\item[(b)] The coefficients \(a^n \colon \mathbb{B} \to S^+ (\mathbb{B}, \mathbb{B})\) have decompositions \(a^n = \sigma^n (\sigma^n)^*\), where \(\sigma^n \colon \mathbb{B} \to L(\mathbb{H}, \mathbb{B})\) is Borel.
\end{enumerate}
Next, we define conditions on the coefficients \((b^n)_{n \in \mathbb{N}}\) and  \((a^n)_{n \in \mathbb{N}}\). 
\begin{enumerate}
	\item[(A1)] For all bounded sets \(G \subset \mathbb{B}\)
	\begin{align}\label{eq: tr}
		\sup_{n \in \mathbb{N}} \sup_{x \in G} \|b^n(x)\| + \sup_{n \in \mathbb{N}} \sup_{x \in G} \|\sigma^n(x)\|_\textup{HS} < \infty.
	\end{align}
	\item[(A2)]
	There exists a constant \(K > 0\) such that for all \(x \in \mathbb{B}\) and all \(n \in \mathbb{N}\)
	\begin{align}\label{eq: tr2}
	\|b^n(x)\| + \|\sigma^n(x)\|_\textup{HS} \leq K \big(1 + \|x\|\big).
	\end{align}
	\item[(A3)]
	For all bounded sets \(G \subset\mathbb{B}\)
	\begin{align}\label{eq: no tr}
		\sup_{n \in \mathbb{N}} \sup_{x \in G} \|b^n(x)\| + \sup_{n \in \mathbb{N}} \sup_{x \in G} \|\sigma^n(x)\|_o < \infty, 
	\end{align}
	where \(\|\cdot\|_o\) denotes the operator norm.
	\item[(A4)]	There exists a constant \(K > 0\) such that for all \(x \in \mathbb{B}\) and all \(n \in \mathbb{N}\)
	\begin{align}\label{eq: no tr2}
	\|b^n(x)\| + \|\sigma^n(x)\|_o \leq K \big(1 + \|x\|\big).
	\end{align}
\end{enumerate}
It is clear that 
\begin{align*}
\textup{(A2)} &\Rightarrow \textup{(A1)},\\ \textup{(A4)} &\Rightarrow \textup{(A3)}.
\end{align*}

Recall 
that a \(C_0\)-semigroup \((S_t)_{t \geq 0}\) is called compact if \(S_t\) is a compact operator for all \(t > 0\). Note that a \(C_0\)-semigroup with generator \(A\) is compact if and only if it is continuous on \((0, \infty)\) in the uniform operator topology and the resolvent of \(A\) is compact, see \cite[Theorem 3.3, p. 48]{pazy2012semigroups}. In particular, by \cite[(2.5), p. 235]{pazy2012semigroups}, an analytic semigroup (see \cite[Definition 5.1, p. 60]{pazy2012semigroups}) whose generator has a compact resolvent is compact.

We also formulate conditions on the linearity \(A\):
\begin{enumerate}
	\item[\textup{(A5)}] The operator \(A\) is the generator of a compact \(C_0\)-semigroup.
	\item[\textup{(A6)}] The operator \(A\) is the generator of a \(C_0\)-semigroup \((S_t)_{t \geq 0}\) and there exists a \(\lambda \in (0, \frac{1}{2})\) and an \(\epsilon > 0\) such that
	\begin{align}\label{eq: convSG}
	\int_0^\epsilon t^{- 2 \lambda} \|S_t\|^2_\textup{HS} \dd t < \infty.
	\end{align}
\end{enumerate}
\begin{lemma}\label{lem: for all conv}
	\begin{enumerate}
		\item[\textup{(i)}]
	\textup{(A6)} \(\Rightarrow\) \textup{(A5)}. 
		\item[\textup{(ii)}] If \eqref{eq: convSG} holds for some \(\epsilon > 0\), then \eqref{eq: convSG} holds for all \(\epsilon > 0\).
\end{enumerate}
\end{lemma}
\begin{proof}
	(i). Fix \(t > 0\) and note that, due to \eqref{eq: convSG}, we find an \(s \in(0, t)\) such that \(S_s\) is Hilbert-Schmidt. Thus, the operator \(S_t\) is Hilbert-Schmidt, because it is the product of the bounded operator \(S_{t - s}\) and the Hilbert-Schmidt operator \(S_s\).
	We conclude that \(S_t\) is compact.
	
	(ii). Suppose that \eqref{eq: convSG} holds for \(\epsilon >  0\) and let \(T > \epsilon\). There exists an \(s \in (0, \epsilon)\) such that \(S_s\) is Hilbert-Schmidt. Moreover, because \((S_t)_{t \geq 0}\) is a \(C_0\)-semigroup, there are constants \(M \geq 1\) and \(\alpha \geq 0\) such that 
	\(
	\|S_t\|_o \leq M e^{\alpha t}
	\) for all \(t \geq 0\), see \cite[Theorem 2.2, p. 4]{pazy2012semigroups}.
	We obtain that 
	\begin{align*}
	\int_0^T t^{-2 \alpha} \|S_t\|^2_\textup{HS} \dd t 
	&\leq  \int_0^\epsilon t^{-2 \alpha} \|S_t\|^2_\textup{HS} \dd t + \int_\epsilon^T t^{-2 \alpha} \|S_s\|^2_\textup{HS} M^2 e^{2\alpha (t - s)} \dd t
	\\&\leq  \int_0^\epsilon t^{-2 \alpha} \|S_t\|^2_\textup{HS} \dd t + (T - \epsilon) \epsilon^{- 2 \alpha} \|S_s\|^2_\textup{HS} M^2 e^{2 \alpha (T - s)} < \infty.
	\end{align*}
	This completes the proof.
	\end{proof}
\subsection{A Tightness Condition}\label{sec: 5.2}
Next, we study tightness of the sequence \((P^n \circ X^{-1}_{\cdot \wedge \tau_m})_{n \in \mathbb{N}}\) when \(P^n \in \mathcal{M}(A, b^n, a^n, \eta^n)\).
A proof for the following proposition can be found in Section \ref{proof: prop1} below.
\begin{proposition}\label{prop: tight}
	Let \(m \in [0, \infty]\) and assume one of the following:
		\begin{enumerate}
		\item[\textup{(i)}] \(m < \infty\), \textup{(A1)} and \textup{(A5)} hold.
		\item[\textup{(ii)}] \(m < \infty\), \textup{(A3)} and \textup{(A6)} hold.
		\item[\textup{(iii)}] \(m = \infty\), \textup{(A2)} and \textup{(A5)} hold.
		\item[\textup{(iv)}] \(m = \infty\), \textup{(A4)} and \textup{(A6)} hold.
		\end{enumerate}
If \(P^n \in \mathcal{M}(A, b^n, a^n, \eta^n)\) and the sequence \((\eta^n)_{n \in \mathbb{N}}\) is tight, then \((P^n \circ X_{\cdot \wedge \tau_m}^{-1})_{n \in \mathbb{N}}\) is tight.
\end{proposition}
\begin{remark}\label{rem1}
	\begin{enumerate}
\item[(i)]	Because \(\tau_\infty = \infty\), Proposition \ref{prop: tight} includes a tightness criterion for the global sequence \((P^n)_{n \in \mathbb{N}}\) as well as for the localizations \((P^n \circ X^{-1}_{\cdot \wedge \tau_m})_{n \in \mathbb{N}}\). Part (iii) of Proposition \ref{prop: tight} is known, see \cite[Remark 3.2, Lemma 3.3]{doi:10.1080/07362999708809484}.
\item[(ii)]
	The proof of  Proposition \ref{prop: tight} uses the compactness method from \cite{doi:10.1080/17442509408833868}, which is based on the compactness of the generalized Riemann-Liouville operator and the factorization formula introduced in \cite{doi:10.1080/17442508708833480}. In the presence of jumps, the image space of the generalized Riemann-Liouville operator is not suitable anymore and the factorization formula might fail, see \cite[Section 11.4]{peszat2007stochastic} for comments.
\item[(iii)] Continuity and linear growth conditions for the existence of solutions to MPs can be found in \cite{doi:10.1080/17442509408833868}.
	\end{enumerate}
\end{remark}

\subsection{Corollaries}\label{sec: 5.3}
In this section we collect corollaries to the results from Section \ref{sec: MR} and Proposition \ref{prop: tight}. 
Let \(b \colon \mathbb{B} \to \mathbb{B}\) and \(a \colon \mathbb{B} \to S^+(\mathbb{B}, \mathbb{B})\) be Borel maps such that \(a = \sigma \sigma^*\) for a Borel map \(\sigma \colon \mathbb{B} \to L(\mathbb{H}, \mathbb{B})\).
We formulate the following conditions:
\begin{enumerate}
	\item[(A7)] For all \(y^* \in D(A^*)\) the maps
	\[
	x \mapsto \la b(x), y^*\ra, \la a(x) y^*, y^*\ra
	\]
	are continuous.
	\item[(A8)]
	For all bounded sets \(G \subset \mathbb{B}\)
	\begin{align*}
	 \sup_{x \in G} \|b(x)\| + \sup_{x \in G} \|\sigma(x)\|_\textup{HS} < \infty.
	\end{align*}
	\item[(A9)] For all bounded sets \(G \subset \mathbb{B}\)
	\begin{align*}
	\sup_{x \in G} \|b(x)\| + \sup_{x \in G} \|\sigma(x)\|_o < \infty.
	\end{align*}
\end{enumerate}
\begin{corollary}\label{coro: diffusion Feller}
	Suppose that one of the following holds:
	\begin{enumerate}
		\item[\textup{(i)}] \textup{(A5), (A7)} and \textup{(A8)} hold.
		\item[\textup{(ii)}] \textup{(A6), (A7)} and \textup{(A9)} hold.
	\end{enumerate}
	If for all \(x \in \mathbb{B}\) the MP \((A, b, a, \varepsilon_x)\) has the unique solution \(P_x\), then \(\{P_x, x \in \mathbb{B}\}\) is Feller.
\end{corollary}
\begin{proof}
	Let \((x_n)_{n \in \mathbb{N}} \subset \mathbb{B}\) and \(x \in \mathbb{B}\) such that \(x_n \to x\) as \(n \to \infty\).
Due to Proposition \ref{prop: tight}, the sequence \((P_{x_n} \circ X_{\cdot \wedge \tau_m}^{-1})_{n \in \mathbb{N}}\) is tight for all \(m \in \mathbb{N}\) and it follows from Theorem \ref{prop: loc} that \(P_{x_n} \to P_x\) weakly as \(n \to \infty\). We conclude that the family \(\{P_x, x \in \mathbb{B}\}\) is Feller.
\end{proof}
This observation generalizes the classical result for the finite-dimensional case given in \cite[Corollary 11.1.5]{SV} and it extends the infinite-dimensional result \cite[Corollary 4.4]{EJP2924} by replacing the tightness assumption with explicit conditions implying it.
We also formulate the following condition:
\begin{enumerate}
\item[(A10)] For all \(y^* \in D(A^*)\) we have 
\(
\la b^n , y^*\ra \to \la b, y^*\ra, \la a^n y^*, y^*\ra \to \la a y^*, y^*\ra
\)
as \(n \to \infty\) uniformly on compact subsets of \(\mathbb{B}\). 
\end{enumerate}
\begin{corollary}\label{coro: uni}
		Let \(P^n \in \mathcal{M}(A, b^n, a^n,  \eta^n)\) for all \(n \in \mathbb{N}\) and suppose that one of the following holds:
	\begin{enumerate}
		\item[\textup{(i)}] \textup{(A1), (A5), (A7), (A9)} and \textup{(A10)} hold.
		\item[\textup{(ii)}] \textup{(A3), (A6), (A7), (A9)} and \textup{(A10)} hold.
	\end{enumerate}
	If for all \(x \in \mathbb{B}\) the MP \((A, b, a, \varepsilon_x)\) has a unique solution \(P_x\), then \(x \mapsto P_x\) is Borel, \(\int P_x \eta(\dd x)\) is the unique solution to the MP \((A, b, a, \eta)\) and \(P^n \to \int P_x \eta(\dd x)\) weakly as \(n \to \infty\) whenever \(\eta^n \to \eta\) weakly as \(n \to \infty\).
\end{corollary}
\begin{proof}
	This follows from Proposition \ref{prop: tight} and Theorem \ref{prop: loc}.
\end{proof}
Corollary \ref{coro: uni} is a generalization of the classical finite-dimensional result \cite[Theorem 11.1.4]{SV} and it extends the related infinite-dimensional result \cite[Lemma 4.3]{EJP2924} via explicit conditions for tightness.
Part (i) of Corollary \ref{coro: uni} can be compared to \cite[Theorems 2.1]{doi:10.1080/07362999708809484} and part (ii) can be compared to \cite[Theorem 2.3]{doi:10.1080/07362999708809484}. We stress that time-inhomogeneous SPDEs were studied in \cite{doi:10.1080/07362999708809484}. For the time-homogeneous case the assumptions in \cite[Theorems 2.1 and 2.3]{doi:10.1080/07362999708809484} are the following:
\begin{enumerate}
	\item[(H1)] For all initial laws \(\eta\) the MP \((A, b, a, \eta)\) has at most one solution.
	\item[(H2)] For all \(y \in \mathbb{B}\) the families \(\{\langle b, y\rangle, \langle b^n, y\rangle \colon n \in \mathbb{N}\}\) and \(\{\langle a y, y\rangle, \langle a^n y, y\rangle \colon  n \in \mathbb{N}\}\) are uniformly equicontinuous on all open and bounded subsets of \(\mathbb{B}\).
	\item[(H3)] The convergence in (A10) holds pointwise. 
	\item[(H4)] (A2) respectively (A4) holds and a similar linear growth condition also holds for \(b\) and \(a\).
	\item[(H5)] The initial laws \(\{\eta, \eta^n \colon n \in \mathbb{N}\}\) satisfy a moment condition and \(\eta^n \to \eta\) weakly as \(n \to \infty\).
	\item[(H6)] (A5) respectively a condition comparable to (A6) holds, see \cite[Remark 3]{doi:10.1080/17442509408833868} and \cite[Remark 2.3]{doi:10.1080/07362999708809484}. 
\end{enumerate}	
In view of \cite[Remark 2.1]{doi:10.1080/07362999708809484}, the MP \((A, b, a)\) is well-posed under the assumptions of \cite[Theorems 2.1 and 2.3]{doi:10.1080/07362999708809484}.
Because for equicontinous families the topologies of pointwise and local uniform convergence coincide (see, e.g., \cite[Lemma 11.3.11]{singh2019introduction}), (H2) and (H3) imply (A7) and (A10). 
	In \cite[Remark 2.2]{doi:10.1080/07362999708809484} it was conjectured that the moment assumption in (H5) is not necessary and only required by the method of identifying the limit. Indeed, the martingale problem method does not need such an assumption and, consequently, in Corollary \ref{coro: uni} no moment condition on the initial laws is imposed. Moreover, instead of (H4), Corollary \ref{coro: uni} contains only the weaker local boundedness assumptions (A1) and (A9) or (A3) and (A9).
We formulate one last condition:
\begin{enumerate}
	\item[(A11)] The MP \((A, b, a, \eta)\) has at most one solution.
\end{enumerate}
In the following corollary we replace the well-posedness assumption in Corollary \ref{coro: uni} by a linear growth condition and a uniqueness assumption for the limiting MP.
\begin{corollary}\label{coro: existence diff}
	Let \(P^n \in \mathcal{M}(A, b^n, a^n,  \eta^n)\) for all \(n \in \mathbb{N}\) and suppose that one of the following  holds:
\begin{enumerate}
	\item[\textup{(i)}] \textup{(A2), (A5), (A7), (A9), (A10)} and \textup{(A11)} hold.
	\item[\textup{(ii)}] \textup{(A4), (A6), (A7), (A9), (A10)} and \textup{(A11)} hold.
\end{enumerate}
	If \(\eta^n \to \eta\) weakly as \(n \to \infty\), then there exists a unique solution \(P\) to the MP \((A, b, a, \eta)\) such that \(P^n \to P\) weakly as \(n \to \infty\).
\end{corollary}
\begin{proof}
This follows from Proposition \ref{prop: tight} and Corollary \ref{coro: existence}.
\end{proof}
\begin{remark}
	Condition (A11) holds for instance under local Lipschitz conditions.
	Another possibility to obtain conditions for uniqueness and well-posedness is to use Girsanov's theorem, see \cite{criens18} and Lemma \ref{lem: wp} below.
\end{remark}

\subsection{Examples}\label{sec:5.5}
In the following we present an example for an application of Corollary \ref{coro: diffusion Feller} and examples for cases where the assumptions (A5) and (A6) are satisfied.
\begin{example}\label{exp: heat}
	Assume that \(\mathbb{B} = \mathbb{H}\), that \(\sigma = \textup{Id}\), that (A6) holds and that for all \(x \in \mathbb{B}\)
	\[
	\|b(x)\| \leq \textup{const. } (1 + \|x\|).
	\] 
	Furthermore, assume that \(x \mapsto \la b(x), y\ra\) is continuous for all \(y \in \mathbb{B}\).
	The MP \((A, b, a)\) corresponds to the Cauchy problem
	\[
	\dd Y_t = (A Y_t + b(Y_t))\dd t + \dd W_t,
	\]
	where \(W\) is a cylindrical standard Brownian motion.
	In this case, we have the following:
	\begin{lemma}\label{lem: wp}
		The MP \((A, b, a)\) is well-posed.
	\end{lemma}
	\begin{proof}
		Let \(x \in \mathbb{B}\). The MPs \((A, b, a, \varepsilon_x)\) and \((A, 0, a, \varepsilon_x)\) have solutions due to \cite[Theorem 2]{doi:10.1080/17442509408833868}.
		Because Ornstein-Uhlenbeck processes have a unique law, the MP \((A, 0, a, \varepsilon_x)\) satisfies uniqueness. Now, \cite[Proposition 3.7]{criens18} yields that also the MP \((A, b, a, \varepsilon_x)\) satisfies uniqueness. 
	\end{proof}
	The previous lemma can be compared to \cite[Theorem 13]{daprato2013}, where a similar result is shown. 
	Let \(P_x\) be the unique solution to the MP \((A, b, a, \varepsilon_x)\).
	The following is a consequence of Corollary \ref{coro: diffusion Feller} and Lemma \ref{lem: wp}.
	\begin{corollary}
		The family \(\{P_x, x \in \mathbb{B}\}\) is Feller.
	\end{corollary}
	In \cite{Maslowski1999} it is shown that for all \(t > 0\) and all bounded weakly sequentially continuous functions \(f \colon \mathbb{B} \to \mathbb{R}\) the map \(x \mapsto E_x [f(X_t)]\) is weakly sequentially continuous, too. The proof is based on the observation that the sequential Feller property is preserved by Girsanov's theorem.
\end{example}
\begin{example}\label{ex:2}
	Let \(\mathcal{O}\) be a bounded domain in \(\mathbb{R}^d\) with smooth boundary and \(m \in \mathbb{N}\). We take \(\mathbb{B} \triangleq L^2(\mathcal{O})\). For any multiindex \(\alpha\) with \(|\alpha| \leq 2m\) let \(\gamma_\alpha \colon \textup{cl}_{\mathbb{R}^d} (\mathcal{O}) \to \mathbb{R}\) be a sufficiently smooth function.
	We define the differential operator 
	\[
	A f (x) \triangleq \sum_{|\alpha| \leq 2m} \gamma_\alpha (x) (\partial^\alpha f) (x), \quad x \in \mathcal{O},
	\]
	for 
	\[
	f \in D(A) \triangleq H^{2m} (\mathcal{O}) \cap H^m_0(\mathcal{O}).
	\]
	Moreover, we assume that there exists a constant \(C > 0\) such that
	\[
	(- 1)^{m + 1} \sum_{|\alpha| = 2m} \gamma_\alpha (x) \xi^\alpha   \geq C |\xi|^{2m}
	\]
	for all \(x \in \mathcal{O}\) and \(\xi \in \mathbb{R}^d.\)
	Due to \cite[Theorem 2.7, p. 211]{pazy2012semigroups}, the operator \(A\) generates an analytic semigroup on \(\mathbb{B}\).
	\begin{enumerate}
	\item[(i)] In case \(m =  1\) the resolvent of \(A\) is compact, see \cite[Remark A.28]{DePrato}, and \(A\) generates a compact \(C_0\)-semigroup, i.e. (A5) holds.
	\item[(ii)] In case \(2m > d\) it has been noted in \cite[Example 3]{doi:10.1080/17442509408833868} that there exists a \(\rho \in (0, \frac{1}{2})\) such that \((- A)^{-\rho}\) is Hilbert Schmidt and, due to \cite[Remark 3]{doi:10.1080/17442509408833868}, this implies that (A6) holds.
	\end{enumerate}
In particular, if \(d = 1\) and \(A\) is the Laplacian, then (A6) holds.
\end{example}
\section{Proofs}\label{sec: Proofs}
\subsection{Proof of Theorem \ref{theo1}}\label{proof: theo1}
For \(\alpha \in \Omega\) we introduce the following sets:
\begin{align*}
V(\alpha) &\triangleq \left\{t >0 \colon \tau_{t} (\alpha) < \tau_{t+}(\alpha) \right\},\\
V'(\alpha) &\triangleq \left\{ t > 0 \colon \alpha (\tau_t (\alpha)) \not = \alpha(\tau_t(\alpha)-) \text{ and } \|\alpha(\tau_t(\alpha) - )\| = t\right\}.
\end{align*}
We stress that \(\tau_{t+}\) is well-defined, because \(t \mapsto \tau_t\) is increasing.
Due to \cite[Problem 13, p. 151]{EK} and \cite[Propositions VI.2.11]{JS}, the map \(\alpha \mapsto \tau_t(\alpha)\) is continuous at each point \(\alpha\) such that \(t \not \in V(\alpha)\). Furthermore, using \cite[Theorem 3.6.3, Remark 3.6.4]{EK} instead of \cite[Theorem VI.1.14 b)]{JS}, we can argue as in the proof of \cite[Proposition VI.2.12]{JS} and obtain that the map \(\alpha \mapsto\alpha (\cdot \wedge \tau_t(\alpha))\) is continuous at each point \(\alpha\) such that \(t \not \in V(\alpha) \cup V'(\alpha)\). 
It follows as in the proof of \cite[Proposition IX.1.17]{JS} that the set
\[
\left\{t \geq 0 \colon P \big(\big\{ \alpha \in \Omega \colon t \in V (\alpha) \cup V'(\alpha)\big\}\big) > 0 \right\}
\]
is at most countable. Therefore, we can choose \(\lambda_m \in [m-1, m]\) such that 
\begin{align}\label{lambda_m}
P\big(\big\{\alpha \in \Omega \colon \lambda_m \in V (\alpha) \cup V'(\alpha)\big\}\big) = 0.
\end{align}
We summarize:
\begin{align}\label{eq: SK con}
\alpha \mapsto \tau_{\lambda_m}(\alpha), X_{\cdot \wedge \tau_{\lambda_{m}}(\alpha)}(\alpha) \text{ are continuous up to a } P\text{-null set.} 
\end{align}
Next, we show that the process \(M^f_{\cdot \wedge \tau_{\lambda_{m}}}\) is a \(P\)-martingale for all \(f \in \mathcal{D}\). Fix an \(f \in \mathcal{D}\) and let \((f^n)_{n \in \mathbb{N}}\) be the corresponding sequence as in (M2).
Define \(M^{f^n}\) as in \eqref{f - K} with \(f\) replaced by \(f^n\) and \(\mathcal{K} f\) replaced by \(\mathcal{K}^n f^n\). Due to Proposition \ref{prop: A1} in Appendix \ref{app: 1}, the process \(M^{f^n}_{\cdot \wedge \tau_{\lambda_{m}}}\) is a \(P^{n}\)-martingale. 
We claim the following: There exists a dense set \(D \subset \mathbb{R}_+\) such that the following hold:
\begin{enumerate}
	\item[\textup{(a)}] For any bounded sequence \((t_n)_{n \in \mathbb{N}}\subset D\) and any \(t \in D\) we have \[\sup_{n \in \mathbb{N}} \sup_{\alpha \in \Omega}\big|M^f_{t_n \wedge \tau_{\lambda_{m}} (\alpha)}(\alpha)\big| + \sup_{n \in \mathbb{N}} \sup_{\alpha \in \Omega} \big|M^{f^n}_{t \wedge \tau_{\lambda_m (\alpha)}}(\alpha)| < \infty.\] 
	\item[\textup{(b)}] For all \(t\in D\) the map \(\alpha \mapsto M^f_{t \wedge \tau_{\lambda_{m}}(\alpha)} (\alpha)\) is continuous up to a \(P\)-null set.
	\item[\textup{(c)}] For all \(t\in D\) and all compact sets \(K \subset \Omega\) we have 
	\begin{align*}
	\sup_{\alpha \in K} \left|M^{f^n}_{t \wedge \tau_{\lambda_{m}} (\alpha)} (\alpha) - M^f_{t \wedge \tau_{\lambda_{m}(\alpha)}} (\alpha)\right| \to 0
	\end{align*}
	as \(n \to \infty\).
\end{enumerate}
Before we check these properties, we show that they imply that the process \(M^f\) is a local \(P\)-martingale.
Let \(t \in D\) and \(k \colon \Omega \to \mathbb{R}\) be bounded and continuous. In this case, (a) and (b) imply that the map
\[\alpha \mapsto k(\alpha) M^f_{t \wedge \tau_{\lambda_{m}}(\alpha)} (\alpha)\] is bounded and continuous up to a \(P\)-null set.
Therefore, by the continuous mapping theorem, \(P^n \to P\) weakly as \(n \to \infty\) yields that 
\[
E^{P^n} \left[k M^f_{t \wedge \tau_{\lambda_{m}}}\right] \to E^{P} \left[k M^f_{t \wedge \tau_{\lambda_{m}}}\right]
\]
as \(n \to \infty\). 
Fix \(\varepsilon > 0\) and denote 
\[
c \triangleq \max \Big( \sup_{\alpha \in \Omega} \big|k(\alpha)\big|,\ \sup_{\alpha \in \Omega} \big| k(\alpha) M^f_{t \wedge \tau_{\lambda_m}(\alpha)}(\alpha)\big|,\ \sup_{n \in \mathbb{N}} \sup_{\alpha \in \Omega} \big| k(\alpha) M^{f^n}_{t \wedge \tau_{\lambda_m}(\alpha)}(\alpha)\big|,\ 1\Big).
\]
Note that \(1 \leq c < \infty\) by (a).
 Because \((P^n)_{n \in \mathbb{N}}\) is tight, we find a compact set \(K \subset \Omega\) such that 
 \[
 \sup_{n \in \mathbb{N}} P^n(K^c) \leq \frac{\varepsilon}{4 c}.
 \]
 Using (c) we find an \(N \in \mathbb{N}\) such that for all \(n \geq N\)
 \begin{align*}
 E^{P^n} \left[ \left|kM^{f^n}_{t \wedge \tau_{\lambda_{m}}} - k M^{f}_{t \wedge \tau_{\lambda_{m}}}\right| \right] &\leq 2c \sup_{m \in \mathbb{N}} P^m (K^c) + c \sup_{\alpha \in K} \big| M^{f^n}_{t \wedge \tau_{\lambda_{m}}(\alpha)}(\alpha) - M^{f}_{t \wedge \lambda_{m}(\alpha)}(\alpha)\big|
 \\&\leq \frac{\varepsilon}{2} + \frac{\varepsilon}{2} = \varepsilon.
 \end{align*}
Therefore, using the triangle inequality, we obtain that
\begin{align*}
E^{P^n}\left[ kM^{f^n}_{t \wedge \tau_{\lambda_{m}}}\right] \to  E^P \left[k M^{f}_{t \wedge \tau_{\lambda_{m}}}\right] 
\end{align*}
as \(n \to \infty\).

Take \(s, t \in \mathbb{R}_+\) with \(s < t\). Since \(D\) is dense in \(\mathbb{R}_+\) we find two sequences \((s_n)_{n \in \mathbb{N}}, (t_n)_{n \in \mathbb{N}}\subset D\) with \(s_n \searrow s\) and \(t_n \searrow t\) as \(n \to \infty\).
For any bounded, continuous and \(\mathscr{F}_s\)-measurable function \(k \colon \Omega \to \mathbb{R}\) we have 
\begin{align*}
E^{P} \left[ k \left(M^{f}_{t \wedge \tau_{\lambda_{m}}} - M^f_{s \wedge \tau_{\lambda_{m}}}\right)\right] &=\lim_{i \to \infty} E^{P} \left[ k \left(M^{f}_{t_i \wedge \tau_{\lambda_{m}}} - M^f_{s_i \wedge \tau_{\lambda_{m}}}\right)\right]
\\&= \lim_{i \to \infty} \lim_{n \to \infty} E^{P^n} \left[ k \left(M^{f^n}_{t_i \wedge \tau_{\lambda_{m}}} - M^{f^n}_{s_i \wedge \tau_{\lambda_{m}}}\right)\right] 
\\&= \lim_{n \to \infty} E^{P^n} \left[ k \left(M^{f^n}_{s \wedge \tau_{\lambda_{m}}} - M^{f^n}_{s \wedge \tau_{\lambda_{m}}}\right)\right] = 0, 
\end{align*}
by the dominated convergence theorem, which we can apply due to (a), the right-continuity of \(M^f_{\cdot \wedge \tau_{\lambda_{m}}}\) and the \(P^n\)-martingale property of \(M^{f^n}_{\cdot \wedge \tau_{\lambda_{m}}}\). We claim that this already proves that \(M^f_{\cdot \wedge \tau_{\lambda_{m}}}\) is a \(P\)-martingale. 

Take \(g \in C_b(\mathbb{B})\) and let \((m_k)_{k \in \mathbb{N}} \subset (0, \infty)\) be such that \(m_k \searrow 0\) as \(k \to \infty\). We set
\[
g^k (q) \triangleq \frac{1}{m_k} \int_{q}^{q+ m_k} g(X_r)\dd r,\quad k \in \mathbb{N}, q \in \mathbb{R}_+,
\]
and note that \(g^k (q) \colon \Omega \to \mathbb{R}\) is continuous, bounded and \(\mathcal{F}_{q + m
	_k}\)-measurable and that
\(g^k(q) \to g(X_q)\) as \(k \to \infty\). Thus, using an approximation argument, we can deduce from the fact that \(E^{P} \big[ k M^f_{t \wedge \tau_{\lambda_m}} \big] = E^{P} \big[ k M^f_{s \wedge \tau_{\lambda_m}}  \big]\) holds for all \(s < t\) and all continuous, bounded and \(\mathcal{F}_s\)-measurable \(k\) that
\[
E^{P} \Big[ M^f_{t \wedge \tau_{\lambda_m}} \prod_{i = 1}^l g_i(X_{q_i}) \Big]
=  E^{P} \Big[ M^{f}_{s \wedge \tau_{\lambda_m}} \prod_{i = 1}^l g_i(X_{q_i}) \Big],
\]
for all \(s < t\), \(l \in \mathbb{N}\), \(g_1, \dots, g_l\in C_b(\mathbb{B})\) and \(q_1, \dots, q_l \in [0, s]\).
Using a monotone class argument shows that \(M^f_{\cdot \wedge \tau_{\lambda_m}}\) is a \(P\)-martingale.

Since \(\lambda_m \in [m-1, m]\), we have \(\lambda_m \nearrow \infty\) as \(m \to \infty\) and therefore also \(\tau_{\lambda_{m}} (\alpha) \nearrow \infty\) as \(m \to \infty\) for all \(\alpha \in \Omega\). In other words, \(M^f\) is a local \(P\)-martingale. 

Due to Proposition \ref{prop: A2} in Appendix \ref{app: 1}, solutions to MPs are determined by the test functions in \(\mathcal{D}\).
Thus, to conclude that \(P \in \mathcal{M}(A, b, a, K, \eta)\) it remains to show that \(P \circ X_0^{-1} = \eta\). Because \(\alpha \mapsto \alpha(0)\) is continuous, we have
\[
\eta^n = P^n \circ X_0^{-1} \to P \circ X_0^{-1}
\]
weakly as \(n \to \infty\). The uniqueness of weak limits and (M3) yield the identity \(P \circ X_0^{-1} = \eta\).
Consequently, \(P \in \mathcal{M}(A, b, a, K, \eta)\) and the theorem is proven.

It remains to check (a) -- (c). The finiteness of the first term in (a) follows from Proposition \ref{prop: A1} in Appendix \ref{app: 1}. The second term is finite due to \eqref{eq: bdd assp} in (M2).

Next, we check (b). 
Set 
\[
D \triangleq \left\{t \geq 0 \colon P\big(X_{t \wedge \tau_{\lambda_{m}}} \not = X_{(t \wedge \tau_{\lambda_{m}})-}\big) = 0\right\}.
\]
By \cite[Lemma 3.7.7]{EK}, the complement of \(D\) in \(\mathbb{R}_+\) is countable. Thus, \(D\) is dense in \(\mathbb{R}_+\). For each \(t \in D\) set 
\[
U_t \triangleq  \big\{\alpha \in \Omega \colon \alpha (t \wedge \tau_{\lambda_{m}}(\alpha)) \not = \alpha((t \wedge \tau_{\lambda_{m}}(\alpha))-) \big\},
\]
which is a \(P\)-null set by the definition of \(D\).
Let \(N \in \mathscr{F}\) be a \(P\)-null set such that the maps \(\alpha \mapsto \tau_{\lambda_{m}} (\alpha), X_{\cdot \wedge \tau_{\lambda_{m} (\alpha)}}(\alpha)\) are continuous at all \(\alpha \not \in N\), see \eqref{eq: SK con}. Take \(t \in D\) and \(\alpha \not \in N \cup U_t\). Recalling \cite[VI.2.3]{JS}, we see that the functions \(\omega \mapsto \omega(t \wedge \tau_{\lambda_{m}} (\omega))\) and \(\omega \mapsto \omega(0)\) are continuous at \(\alpha\). Thus, \(M^f_{t \wedge \tau_{\lambda_{m}}}\) is continuous at \(\alpha\) if the map
\[
\omega \mapsto I_{t \wedge \tau_{\lambda_m}(\omega)} (\omega) \triangleq \int_0^{t \wedge \tau_{\lambda_{m}}(\omega)} \mathcal{K} f(\omega(s-))\dd s
\]
is continuous at \(\alpha\). 
Because the set of all \(a > 0\) such that \(\tau_a\) is not continuous at \(\alpha\) is at most countable (see \cite[Lemma VI.2.10, Proposition VI.2.11]{JS}) and \(\tau_a (\alpha) \nearrow \infty\) as \(a \to \infty\), we find a \(\hat{\lambda}_m < \infty\) such that \(\lambda_m \leq \hat{\lambda}_m\), \(\tau_{\hat{\lambda}_m}\) is continuous at \(\alpha\) and \(\tau_{\hat{\lambda}_m} (\alpha) > t\). Let \((\alpha_k)_{k \in \mathbb{N}} \subset \Omega\) be such that \(\alpha_k \to \alpha\) as \(k \to \infty\). Due to the continuity of \(\tau_{\hat{\lambda}_m}\) at \(\alpha\), there exists an \(N \in \mathbb{N}\) such that \(\tau_{\hat{\lambda}_m} (\alpha_k) > t\) for all \(k\geq N\).
Due to Proposition \ref{prop: A1} in Appendix \ref{app: 1}, we have \(C \triangleq \sup_{\|x\| \leq \hat{\lambda}_m} |\mathcal{K} f (x)|< \infty\).
Now, for all \(k \geq N\) we have
\begin{align*}
\big| I_{t \wedge \tau_{\lambda_m}(\alpha)} (\alpha) &- I_{t \wedge \tau_{\lambda_m}(\alpha_k)} (\alpha_k)\big| \\&\leq \big| I_{t \wedge \tau_{\lambda_m}(\alpha)} (\alpha) - I_{t \wedge \tau_{\lambda_m}(\alpha)} (\alpha_k)\big| + \big| I_{t \wedge \tau_{\lambda_m}(\alpha)} (\alpha_k) - I_{t \wedge \tau_{\lambda_m}(\alpha_k)} (\alpha_k)\big|
\\&\leq \big| I_{t \wedge \tau_{\lambda_m}(\alpha)} (\alpha) - I_{t \wedge \tau_{\lambda_m}(\alpha)} (\alpha_k)\big| + C \big|t \wedge \tau_{\lambda_m}(\alpha) - t \wedge \tau_{\lambda_m}(\alpha_k)\big|.
\end{align*}
Because for all \(k \geq N\) and all \(s \leq t \wedge \tau_{\lambda_m} (\alpha) < \tau_{\hat{\lambda}_m}(\alpha) \wedge \tau_{\hat{\lambda}_m}(\alpha_k)\) we have 
\[
\big| \mathcal{K} f (\alpha(s-)) - \mathcal{K} f(\alpha_k(s-))\big| \leq 2 C, 
\]
the first term goes to zero as \(k \to \infty\) by (M1), the dominated convergence theorem and \cite[VI.2.3, Lemma VI.3.12]{JS}. The second term goes to zero as \(k \to \infty\) by the continuity of \(\tau_{\lambda_{m}}\) at \(\alpha\). 
We conclude that (b) holds.

Finally, we check (c). Due to \cite[Problem 16, p. 152]{EK}, for each compact set \(K \subset \Omega\) and each \(t \geq 0\) there exists a compact set \(K_t \subset \mathbb{B}\) such that \(\alpha (s) \in K_t\) for all \(\alpha \in K\) and \(s \in [0, t]\). Thus, (M2) implies that
\begin{align*}
\sup_{\alpha \in K} \Big| M^{f^n}_{t \wedge \tau_{\lambda_{m}}(\alpha)} &(\alpha)- M^f_{t \wedge \tau_{\lambda_{m}}(\alpha)} (\alpha)\Big|  
\\&\leq 2 \sup_{x \in K_t} \big|f^n(x) - f(x)\big| +  t \sup_{x \in K_t}\big| \mathcal{K}^nf^n (x) - \mathcal{K} f(x)\big| \to 0
\end{align*} as \(n \to \infty\).
Therefore, (c) holds and the proof is complete.
\qed
\subsection{Proof of Theorem \ref{prop: loc}}\label{proof: theo2}
	Proposition \ref{prop: A3} in Appendix \ref{app: 1} yields that the map \(x \mapsto P_x\) is Borel and that \(P \triangleq \int P_x \eta(\dd x)\) is the unique solution to MP \((A, b, a, K, \eta)\). We show that \(P^n \to P\) weakly as \(n \to \infty\). It is well-known that \(P^n \to P\) weakly as \(n \to \infty\) if and only if each subsequence of \((P^n)_{n \in \mathbb{N}}\) has a further subsequence which converges weakly to \(P\), see, e.g., \cite[Theorem 2.6]{bil99}. If we show that \((P^n)_{n \in \mathbb{N}}\) is tight, Prohorov's theorem yields that any subsequence of \((P^n)_{n \in \mathbb{N}}\) has a weakly convergent subsequence, and Theorem \ref{theo1}, together with the uniqueness of \(P\), yields that this subsequence converges weakly to \(P\). In summary, it suffices to prove that \((P^n)_{n \in \mathbb{N}}\) is tight.
	
We define the following modulus of continuity:
\begin{align*}
w'(\alpha, \theta, N) &\triangleq \inf_{\{t_i\}}  \max_{i} \sup_{s, t \in [t_{i-1}, t_i)} \|\alpha(s) - \alpha(t)\|,
\end{align*}
where \(\{t_i\}\) ranges over all partitions of the form \(0 = t_0 < t_1 < \dots < t_{n -1} < t_n \leq N\) with \(\min_{1 \leq i < n}(t_i - t_{i-1}) \geq \theta\) and \(n \geq 1\).
Now, recall the following fact (see \cite[Corollary 3.7.4]{EK}):
\begin{fact}\label{fact}
	A sequence \((\mu^n)_{n \in \mathbb{N}}\) of probability measures on \((\Omega, \mathscr{F})\) is tight if and only if the following hold:
	\begin{enumerate}
		\item[\textup{(a)}]
		For all \(t \in \mathbb{Q}_+\) and \(\epsilon > 0\) there exists a compact set \(C(t, \epsilon) \subset \mathbb{B}\) such that
		\[
		\limsup_{n \to \infty} \mu^n(X_t \not \in C(t, \epsilon))\leq \epsilon. \]
		\item[\textup{(b)}]
		For all \(\epsilon > 0\) and \(t > 0\) there exists a \(\delta > 0\) such that 
		\[
		\limsup_{n \to \infty} \mu^n(w'(X, \delta, t) \geq \epsilon) \leq \epsilon.
		\]
	\end{enumerate}
\end{fact}
In the remainder of this proof, we show that \((P^n)_{n \in \mathbb{N}}\) satisfies (a) and (b) in Fact \ref{fact}. We start with a few preparations.
In what follows let \(m \in \mathbb{N}\) be arbitrary. Denote \(P^{n, m} \triangleq P^n \circ X^{-1}_{\cdot \wedge \tau_m}\) and fix \(t \geq 0\). 
Let \(Q^m\) be an accumulation point of \((P^{n, m})_{n \in \mathbb{N}}\).
Using the same arguments as in the proof of Theorem \ref{theo1}, we find a \(\lambda_m \in [m-1, m]\) such that for all \(f \in \mathcal{D}\) the process \(M^f_{\cdot \wedge \tau_{\lambda_{m}}}\) is a \(Q^m\)-martingale and \(Q^m \circ X^{-1}_0 = \eta\).
Due to Proposition \ref{prop: A3} in Appendix \ref{app: 1}, the assumption that the MP \((A, b, a, K)\) is well-posed implies that
\begin{align*} 
Q^m = P \text{ on } \mathscr{F}_{\tau_{\lambda_{m}}}.
\end{align*}
We note that the choice of \(\lambda_m\) depends on \(Q^m\), see \eqref{lambda_m}. However, for any accumulation point of \((P^{n, m})_{n \in \mathbb{N}}\) we find an appropriate \(\lambda_m\) in the interval \([m-1, m]\). 
Thus, any accumulation point of \((P^{n, m})_{n \in \mathbb{N}}\) coincides with \(P\) on \(\mathscr{F}_{\tau_{m-1}}\).
Due to \cite[Problem 13, p. 151]{EK} and \cite[Lemma 15.20]{HWY}, the map \[\alpha \mapsto M^*_t (\alpha) \triangleq \sup_{s \in [0, t]} \|\alpha(s)\|\] is upper semicontinuous. Thus, the set 
\[
\{\tau_{m-1} \leq t\} = \{M^*_t \geq m-1\}
\]
is closed in the Skorokhod topology. We deduce from the Portmanteau theorem that 
\begin{equation}\label{port}\begin{split}
 \limsup_{n \to \infty}\ P^{n, m} (\tau_{m-1} \leq t) \leq P (\tau_{m-1} \leq t).
\end{split}\end{equation}
Fix \(\epsilon > 0\).
Since \(P(\tau_{m-1} \leq t) \searrow 0\) as \(m \to \infty\), we find an \(m^o \in \mathbb{N}_{\geq 2}\) such that 
\begin{align}\label{eq: P bound}
P(\tau_{m^o -1} \leq t) \leq \tfrac{\epsilon}{2}.
\end{align}
Because \((P^{n, m^o-1})_{n \in \mathbb{N}}\) is tight, we deduce from Fact \ref{fact} that there exists a compact set \(C(t, \epsilon) \subset \mathbb{B}\) such that 
\begin{align}\label{eq: t imp}
\limsup_{n \to \infty} P^n (X_{t \wedge \tau_{m^o-1}} \not \in C(t, \epsilon)) \leq \tfrac{\epsilon}{2}.
\end{align}
Due to Galmarino's test, see \cite[Lemma III.2.43]{JS}, we have \(\tau_{m^o-1} = \tau_{m^o-1} \circ X_{\cdot \wedge \tau_{m^o}}\). Thus, we obtain
\begin{align*}
P^n (X_t \not \in C(t, \epsilon)) & = P^n(X_{t \wedge \tau_{m^o-1}} \not \in C(t, \epsilon), \tau_{m^o-1} > t) + P^n(X_t \not \in C(t, \epsilon), \tau_{m^o-1} \leq t)
\\&\leq P^n (X_{t \wedge \tau_{m^o-1}} \not \in C(t, \epsilon)) + P^{n, m^o}(\tau_{m^o-1} \leq t).
\end{align*}
From this, \eqref{port}, \eqref{eq: P bound} and \eqref{eq: t imp}, we deduce that 
\[
\limsup_{n \to \infty} P^n(X_t \not \in C(t, \epsilon)) \leq \epsilon.
\]
This proves that the sequence \((P^n)_{n \in \mathbb{N}}\) satisfies (a) in Fact \ref{fact}. 

Next, we show that \((P^n)_{n \in \mathbb{N}}\) satisfies (b) in Fact \ref{fact}. Let \(\epsilon, t\) and \(m^o\) be as before. Because \((P^{n, m^o-1})_{n \in \mathbb{N}}\) is tight there exists a \(\delta > 0\) such that 
\begin{align}\label{eq: t imp2}
\limsup_{n \to \infty} P^n \left(w' (X_{\cdot \wedge \tau_{m^o-1}}, \delta, t) \geq \epsilon\right) \leq \tfrac{\epsilon}{2}.
\end{align}
On the set \(\{\tau_{m^o-1} > t\}\) we have \(w' (X, \delta, t) = w'(X_{\cdot \wedge \tau_{m^o-1}}, \delta, t)\).
Thus, using \eqref{port}, \eqref{eq: P bound} and \eqref{eq: t imp2}, we obtain 
\begin{align*}
\limsup_{n \to \infty}&\ P^n (w'(X, \delta, t) \geq \epsilon) \\&\leq \limsup_{n \to \infty}P^n (w'(X_{\cdot \wedge \tau_{m^o-1}}, \delta, t) \geq \epsilon) + \limsup_{n \to \infty} P^{n, m^o} (\tau_{m^o-1} \leq t) \leq \epsilon.
\end{align*}
In other words, \((P^n)_{n \in \mathbb{N}}\) satisfies (b) in Fact \ref{fact} and the proof is complete.
\qed

\subsection{Proof of Theorem \ref{prop: existence}}\label{sec: pf sm}
For each \(n \in \mathbb{N}\) let \(\phi^n \colon \mathbb{R}^n \to [0, 1]\) be the standard mollifier on \(\mathbb{R}^n\), see, e.g., \cite[p. 147]{gilbarg2001elliptic}. Recall that \(\phi^n\) is supported on the Euclidean unit ball. Set \(\mlambda \triangleq \sup_{k \in \mathbb{N}} \lambda_k < \infty\).
We fix a sequence \((\epsilon_n)_{n \in \mathbb{N}} \subset (0, \infty)\) such that \(\epsilon_n \leq \tfrac{1}{n} \min_{k = 1, \dots, n} \lambda_k\) for all \(n \in \mathbb{N}\). Clearly, we have \(\epsilon_n \leq \frac{1}{n}\mlambda \to 0\) as \(n \to \infty\).
Define 
\[
\theta_n x \triangleq \big(\langle x, h_1\bs, \dots, \langle x, h_n\bs \big), \quad x \in \mathbb{B}, 
\]
and set 
\[
v^n (y, x) \triangleq \frac{1}{\epsilon_n^n}\int_{\mathbb{R}^n} \phi^n \Big(\frac{z - \theta_n x}{\epsilon_n}\Big)v\Big(y, \sum_{i = 1}^n \frac{z^i e_i}{\lambda_i}\Big) \dd z, \quad y \in E, x \in \mathbb{B}.
\]
We define \(b^n\) and \(\sigma^n\) in the same manner. 
Here, recall that \((h_k)_{k \in \mathbb{N}}\) is an orthonormal basis of \(\mathbb{B}\) and that \((e_k)_{k \in \mathbb{N}} = (\lambda_k h_k)_{k \in \mathbb{N}}\) is an orthonormal basis of \(\mathbb{K} = J(\mathbb{B})\).
Next, we check properties of \(v^n, b^n\) and \(\sigma^n\). 
\begin{lemma}\label{lem: 1}
For all \(m \in \mathbb{N}\) and all \(x, z \in \mathbb{K}\) with \(\|x\kn, \|z\kn  \leq m\) there exist constants \(L = L(n, m) \in (0, \infty)\) and \(l = l(n, m) \in (0, \infty)\) such that \[
\int_E \|v^n (y, x) - v^n(y, z)\kn^2 F(\dd y) \leq L \|x - z\kn^2
\]
and 
\[
\|b^n (x) - b^n(z)\kn + \|\sigma^n(x) - \sigma^n(z)\|_{\textup{HS} (\mathbb{K})} \leq l \|x - z\kn.
\]
\end{lemma}
\begin{proof}
We only prove the claim for \(v^n\). For \(b^n\) and \(\sigma^n\) it follows in the same manner.
	Fix \(m \in \mathbb{N}\) and let \(y \in E\) and \(x, z \in \mathbb{K}\) with \(\|x\kn, \|z\kn \leq m\). Define \[G (n, m) = G \triangleq \{u \in \mathbb{R}^n \colon \|u\dn < \|\iota\|_o m + \epsilon_n\},\]
where \(\|\cdot\dn\) denotes the Euclidean norm on \(\mathbb{R}^n\) and \(\|\iota\|_o\) denotes the operator norm of \(\iota\).
Using the Parseval identity, we obtain
\[
\|\theta_n x \dn^2 = \sum_{i = 1}^n \langle x, h_i\bs^2 \leq \|x\bn^2 \leq \|\iota\|^2_o \|x\kn^2.
\]
For all \(u \not \in G\) we have 
\[
\| u  - \theta_n x\dn \geq \|u\dn - \|\theta_n x\dn \geq \|u\dn - \|\iota\|_o m \geq \epsilon_n.
\]
Because smooth functions with compact support are Lipschitz continuous, the function \(\phi^n\) is Lipschitz continuous and we denote the corresponding Lipschitz constant by \(L = L(n)\). 
Furthermore, we have for all \(u \in G\)
\[
\Big\| \sum_{i = 1}^n \frac{u^i}{\lambda_i} e_i\Big\|_\mathbb{K} \leq \sum_{i = 1}^n \frac{\|\iota\|_o m + \epsilon_n}{\lambda_i} \triangleq C = C(n, m).
\]
Now, we deduce from the linear growth assumption, i.e. hypothesis (iii), that
\begin{align*}
\|v^n (y, x) - v^n(y, z)\kn &\leq \frac{1}{\epsilon_n^n} \int_{\mathbb{R}^n} \Big|\phi^n\Big(\frac{u - \theta_n x}{\epsilon_n}\Big) - \phi^n\Big(\frac{u - \theta_n z}{\epsilon_n}\Big)\Big|  \Big\|v\Big(y, \sum_{i = 1}^n \frac{u^i}{\lambda_i} e_i\Big)\Big\|_\mathbb{K} \dd u
\\&=\frac{1}{\epsilon_n^n} \int_{G} \Big|\phi^n\Big(\frac{u - \theta_n x}{\epsilon_n}\Big) - \phi^n\Big(\frac{u - \theta_n z}{\epsilon_n}\Big)\Big| \Big\| v\Big(y, \sum_{i = 1}^n \frac{u^i}{\lambda_i} e_i \Big) \Big\|_\mathbb{K} \dd u
\\&\leq \frac{\gamma(y) (1 + C ) L \|\iota\|_o |G| }{\epsilon^{n+1}_n} \|x - z\kn,
\end{align*}
where \(|G|\) denotes the Lebesgue measure of \(G\).
Thus, we have 
\[
\int_E \|v^n (y, x) - v^n(y, z)\kn^2 F(\dd y) \leq \int_E \gamma^2(y)F(\dd y)\Big(\frac{(1 + C ) L \|\iota\|_o |G| }{\epsilon^{n+1}_n}\Big)^2 \|x - z\kn^2.
\]
This completes the proof.
\end{proof}

\begin{lemma}\label{lem: 2}
There exists a constant \(l \in (0, \infty)\) such that for all \(y \in E\) and \(x \in \mathbb{K}\)
\begin{align*}
\|v^n(y, x)\kn^2  &\leq l \gamma^2(y) \big(1 + \|x\kn^2 \big), \\
\|b^n(x)\kn^2 + \|\sigma^n(x)\|^2_{\textup{HS} (\mathbb{K})} &\leq l \big(1 + \|x\kn^2\big).
\end{align*}
\end{lemma}
\begin{proof}
For all \(u \in \mathbb{R}^n\) with \(\|u\dn \leq 1\) the triangle and the Cauchy-Schwarz inequality yield that
\[
\Big\| \sum_{i = 1}^n \frac{\epsilon_n u^i}{\lambda_i} e_i\Big\|_\mathbb{K} \leq \sum_{i = 1}^n \frac{\epsilon_n}{\lambda_i} |u^i| \leq \frac{\|u\dn}{\sqrt{n}} \leq 1.
\]
Thus, we deduce from the linear growth assumption, i.e. hypothesis (iii), that for all \(y \in E\) and \(x \in \mathbb{K}\)
\begin{align*}
\|v^n(y, x)\kn &= \Big\| \int_{\mathbb{R}^n} \phi^n(u) v \Big(y, \sum_{i = 1}^n \Big( \frac{\epsilon_n u^i}{\lambda_i} + \frac{\langle x, h_i\bs}{\lambda_i} \Big) e_i \Big) \Big\|_\mathbb{K}
\\&\leq \int_{\mathbb{R}^n} \phi^n(u)  \Big\|v\Big(y, \sum_{i = 1}^n \Big(\frac{\epsilon_nu^i}{\lambda_i} + \frac{\langle x, h_i\bs}{\lambda_i} \Big) e_i\Big)\Big\|_\mathbb{K} \dd u \\&\leq \gamma(y) \Big(2 + \Big\| \sum_{i = 1}^n \frac{\langle x, h_i\bs}{\lambda_i} e_i\Big\|_\mathbb{K} \Big) \int_{\mathbb{R}^n} \phi^n(u)\dd u
\\&= \gamma(y)\Big(2 + \Big\| \sum_{i = 1}^n \frac{\langle x, h_i\bs}{\lambda_i} e_i\Big\|_\mathbb{K} \Big).
\end{align*}
Note that for \(x \in \mathbb{K}\)
\begin{align*}
\langle x, h_i\bs = \lambda_i \langle J^{-1} x, h_i\bs = \lambda_i \langle J^{-1} x, J^{-1} e_i\bs = \lambda_i \langle x, e_i\ks
\end{align*}
and that
\begin{align*}
\Big(2 + \Big\| \sum_{i = 1}^n \langle x, e_i\ks e_i\Big\|_\mathbb{K} \Big)^2 &\leq 8 + 2 \Big\| \sum_{i = 1}^n \langle x, e_i\ks e_i\Big\|^2_\mathbb{K}
\\&= 8 + 2 \sum_{i = 1}^n \langle x, e_i\ks^2
\\&\leq 8 \big(1 + \|x\kn^2\big),
\end{align*}
where we use the Parseval identity and the elementary inequality \((a_1 + a_2)^2 \leq 2 a^2_1 + 2 a^2_2\) for all \(a_1, a_2 \in \mathbb{R}\).
Thus, we obtain
\[
\|v^n(y, x)\kn^2 \leq 8  \gamma^2(y) \big(1 + \|x\kn^2\big).
\]
A similar argument applies for \(b^n\) and \(\sigma^n\).
\end{proof}
\begin{lemma}\label{lem: 3}
For all \(y \in E\) we have \[
\|v^n(y, \cdot) - v(y, \cdot)\bn + \| b^n - b\bn + \| \sigma^n - \sigma\|_{L(\mathbb{B})} \to 0\]
as \(n \to \infty\) uniformly on compact subsets of \(\mathbb{B}\). Moreover, for any compact set \(K \subset \mathbb{B}\) it holds that
\[
\int_E \sup_{n \in \mathbb{N}} \sup_{x \in K} \| v^n(y, x)\bn^2  F(\dd y) < \infty.
\]
\end{lemma}
\begin{remark}
	We stress that \(y \mapsto \sup_{n \in \mathbb{N}} \sup_{x \in K} \|v^n(y, x)\bn\) is \(\mathcal{E}\)-measurable. To see this, recall that compact metric spaces are separable, i.e. there exists a countable dense (w.r.t. \(\|\cdot\bn\)) set \(K' \subset K\). Because \(x \mapsto v(y, x)\) is continuous for all \(y \in E\), we conclude that the map \[
	y \mapsto \sup_{n \in \mathbb{N}} \sup_{x \in K} \|v^n(y, x)\bn = \sup_{n \in \mathbb{N}} \sup_{x \in K'} \|v^n(y, x)\bn
	\]
	 is \(\mathcal{E}\)-measurable as the countable supremum over \(\mathcal{E}\)-measurable functions.
\end{remark}
\begin{proof}
	Again, we only show the claim for \(v^n\).
Fix \(y \in E\) and \(\varepsilon > 0\) and let \(K \subset \mathbb{B}\) be compact. 
We set 
\[
G_n \triangleq \Big\{  \sum_{i = 1}^n \big(\epsilon_n u^i h_i + \langle x, h_i\bs h_i\big) \colon x \in K, u \in \mathbb{R}^n \text{ with } \|u\dn \leq 1\Big\}.
\]
and \(G \triangleq K \cup \big(\bigcup_{n \in \mathbb{N}} G_n\big)\).
For all \(n \in \mathbb{N}\) the set \(G_n\) is compact in \(\mathbb{B}\) as it is the image of the compact set \(\{u \in \mathbb{R}^n \colon \|u\dn \leq 1\} \times K\) under the continuous map
\[
(u, x) \mapsto \sum_{i = 1}^n \big(\epsilon_n u^i h_i + \langle x, h_i\bs h_i\big).
\]
We claim that also the set \(G\) is compact in \(\mathbb{B}\). To see this take a sequence \((y_n)_{n \in \mathbb{N}}\subset G\). We have to show that \((y_n)_{n \in \mathbb{N}}\) has a subsequence converging to an element in \(G\). There exists a sequence \((k_n)_{n \in \mathbb{N}} \subset \mathbb{N}\) and two sequences \((x_n)_{n \in \mathbb{N}} \subset \mathbb{B}\) and \((u_n)_{n \in \mathbb{N}}\) with \(u_n \in \mathbb{R}^{k_n}\) and \(\|u_n\|_{\mathbb{R}^{k_n}} \leq 1\) such that 
\[
y_n = \sum_{i = 1}^{k_n} \big(\epsilon_{k_n} u^i_n h_i + \langle x_n, h_i\bs h_i\big),\quad n \in \mathbb{N}.
\]
Suppose that \(k \triangleq \sup_{n \in \mathbb{N}} k_n < \infty\). Then, we have 
\(
(y_n)_{n \in \mathbb{N}}\subset\bigcup_{i = 1}^k G_i.
\)
Because \(\bigcup_{i = 1}^k G_i\) is compact as the finite union of compact sets, the sequence \((y_n)_{n \in \mathbb{N}}\) has a subsequence converging to an element in \(\bigcup_{i = 1}^k G_i\subset G\).
Suppose now that \(k = \infty\). Passing to a subsequence if necessary, we can assume that \(k_n \to \infty\) as \(n \to \infty\) and, because \(K\) is compact, we can further assume that there exists an \(x \in K\) such that \(x_n \to x\) as \(n \to \infty\).
Now, we have 
\begin{align*}
\Big\| \sum_{i = 1}^{k_n} \big(\epsilon_{k_n} u^i_n h_i &+ \langle x_n, h_i\bs h_i \big) - x \Big\|_\mathbb{B} \\&\leq \sum_{i = 1}^{k_n} |u^i_n| \epsilon_{k_n} + \Big\| \sum_{i = 1}^{k_n} \langle x_n - x, h_i \bs h_i - \sum_{i = k_n + 1}^\infty \langle x, h_i\bs h_i \Big\|_{\mathbb{B}} \\&\leq \frac{\mlambda}{\sqrt{k_n}} + \|x_n - x\bn + \sqrt{\sum_{i = k_n + 1}^\infty \langle x, h_i\bs^2}
\to 0\end{align*}
as \(n \to \infty\). Thus, in the case \(k = \infty\) the sequence \((y_n)_{n \in \mathbb{N}}\) has a subsequence converging to a point in \(K \subset G\). We conclude the compactness of \(G\).

Because \(G\) is compact, hypothesis (ii) implies that the map \(G \ni x \mapsto v (y, x)\) is uniformly continuous. In other words, there exists a \(\delta > 0\) such that for all \(x, z \in G\) with \(\|x - z\bn < \delta\) we have 
\[
\|v(y, x) - v(y, z)\bn \leq \varepsilon.
\]
Let \(\epsilon \leq \delta (2 \sqrt{2})^{-1}\).
Because compact sets are totally bounded, there exists an \(N_1 \in \mathbb{N}\) and points \(x_1, \dots, x_{N_1} \in \mathbb{B}\) such that 
\[
K \subset \bigcup_{i = 1}^{N_1} B_{x_i} (\epsilon), 
\]
where \(B_{x_i}(\epsilon) \triangleq \{x \in \mathbb{B} \colon \|x - x_i\bn \leq \epsilon\}\).
Take \(u \in \mathbb{R}^n\) with \(\|u\dn \leq 1\) and \(x \in K\). We find a \(k \in \{1, \dots, N_1\}\) such that \(\|x - x_k\bn \leq \epsilon\). Thus, we have
\begin{align*}
\Big\| \sum_{i = 1}^n \big(\epsilon_n u^i h_i + \langle x, h_i\bs h_i \big) - x \Big\|_\mathbb{B}
&\leq \frac{\mlambda}{\sqrt{n}} + \sqrt{\sum_{i = n + 1}^\infty \langle x, h_i\bs^2}
\\&\leq  \frac{\mlambda}{\sqrt{n}}  + \sqrt{2\epsilon^2 + 2\max_{j = 1, \dots, N_1} \sum_{i = n + 1}^\infty \langle x_j, h_i\bs^2}
\\&\leq  \frac{\mlambda}{\sqrt{n}}  + \frac{\delta}{2} + \sqrt{2\max_{j = 1, \dots, N_1} \sum_{i = n + 1}^\infty \langle x_j, h_i\bs^2}.
\end{align*}
Therefore, we find an \(N_2 \in \mathbb{N}\) and for all \(n \geq N_2\) we have 
\[
\sup_{\|u\dn \leq 1} \sup_{x \in K} \Big\| \sum_{i = 1}^n \big(\epsilon_n u^i h_i + \langle x, h_i\bs h_i \big) - x \Big\|_\mathbb{B} < \delta.
\]
We conclude that for all \(n \geq N_2\)
\begin{align*}
\sup_{x \in K} \|v^n (y, x) &- v(y, x)\bn\\&\leq \sup_{x \in K} \int_{\mathbb{R}^n} \phi^n(u) \Big\| v\Big(y, \sum_{i = 1}^n \big(\epsilon_n u^i h_i + \langle x, h_i\bs h_i\big)\Big) - v\Big(y, x\Big)\Big\|_{\mathbb{B}} \dd u \leq \varepsilon.
\end{align*}
This proves the first claim.

To see that the second claim holds it suffices to note that 
\[
\sup_{n \in \mathbb{N}} \sup_{x \in K} \|v^n (y, x)\bn \leq  \sup_{x \in G} \|v(y, x)\bn
\]
and to use hypothesis (iii). The proof is complete.
\end{proof}
\begin{lemma}\label{eq: loc bdd}
	For any bounded set \(G \subset \mathbb{B}\) it holds that
	\[
	\sup_{n \in \mathbb{N}} \sup_{x \in G} \Big( \|b^n(x)\bn + \|\sigma^n(x)\|_{L(\mathbb{B})}  + \int_E \|v^n(x, y)\bn^2 F(\dd y) \Big) < \infty.
	\]
\end{lemma}
\begin{proof}
	There exists an \(m \in \mathbb{N}\) such that for all \(x \in G\) we have \(\|x\bn \leq m\). Thus, for all \(x \in G\) and \(u \in \mathbb{R}^n\) with \(\|u\dn \leq 1\) we have 
	\begin{align*}
	\Big\| \sum_{i = 1}^n \big(\epsilon_n u^i + \langle x, h_i\bs \big)h_i \Big\|^2_\mathbb{B} &\leq 2 \Big\|\sum_{i = 1}^n \epsilon_n u^i h_i\Big\|^2_{\mathbb{B}} + 2 \sum_{i = 1}^n \langle x, h_i\bs^2
	\leq 2\mlambda^2 + 2 m^2.
	\end{align*}
	Set \(G' \triangleq \{x \in \mathbb{B} \colon \|x\bn^2 \leq 2 \mlambda^2 + 2 m^2\}\).
	Thus, using hypothesis (iii), we obtain that 
	\begin{align*}
	\sup_{n \in \mathbb{N}} \sup_{x \in G} \int_E \|v^n(y, x)\bn^2 F(\dd y) 
	\leq \int_E \zeta^2_{G'}(y)F(\dd y) < \infty.
	\end{align*}
	The remaining claims follow in the same manner.
\end{proof}
Denote by \(\Omega^o\) the Skorokhod space of \cadlag functions \(\mathbb{R}_+ \to \mathbb{K}\), equip it with the Skorokhod topology and denote the corresponding Borel \(\sigma\)-field by \(\mathcal{F}^o\).
\begin{lemma}\label{lem: sde}
For all \(n \in \mathbb{N}\) the MP \((0, \bar{b}^{n}, \bar{a}^n, \bar{K}^n, \eta)\) has a solution \(\bar{P}^{n}\), seen as a probability measure on \((\Omega^o, \mathcal{F}^o)\), where for all \(x \in \mathbb{K}\)
\begin{align*}
\bar{b}^{n} (x) &\triangleq b^n (x) + \int_E \big(h(v^n(y, x)) - v^n(y, x)\big)F(\dd y),\\ \bar{a}^n (x)&\triangleq \sigma^n (x) \sigma^n (x)^*, \\ \bar{K}^n(x, G) &\triangleq \int_E \1_{G\backslash \{0\}} (v^n(y, x)) F(\dd y),\quad G \in \mathcal{B}(\mathbb{K}).\end{align*}
Here, \(\sigma^n(x)^* \in L (\mathbb{K}, \mathbb{H})\) denotes the adjoint of \(\sigma (x)\in L(\mathbb{H},\mathbb{K})\).
\end{lemma}
\begin{proof}
	We show that the SDE	\begin{align}\label{eq: SDE loc L}
	\dd Y_t = b^n (Y_{t-}) \dd t + \sigma^n(Y_{t-})\dd W_t + \int_E v^n (y, Y_{t-}) (p - q)(\dd y, \dd t),\quad Y_0 = \zeta \sim \eta,
	\end{align} has a solution with paths in \(\Omega^o\). Fix a filtered probability space which supports a cylindrical standard Brownian motion \(W\) and a compensated random measure \(p - q\).
	As in the proof of \cite[Lemma 34.11]{metivier}, we define for \(m \in \mathbb{N}, y \in E\) and \(x \in \mathbb{K}\)
	\begin{align*}
	b^{n, m}(x) &\triangleq b^n \big(\big(1 \wedge \tfrac{m}{\|x\kn}\big) x\big),\\
	\sigma^{n, m}(x) &\triangleq \sigma^n\big(\big(1 \wedge \tfrac{m}{\|x\kn}\big) x\big),\\
	v^{n, m} (y, x) &\triangleq v^n \big(y, \big(1 \wedge  \tfrac{m}{\|x\kn}\big) x\big).
\end{align*}	
	Due to Lemma \ref{lem: 1}, the coefficients \(b^{n, m}, \sigma^{n, m}\) and \(v^{n, m}\) satisfy a global Lipschitz condition, whose Lipschitz constant might depend on \(m\). Thus, we deduce from \cite[Theorem 3.11]{doi:10.1080/17442501003624407}\footnote{\cite[Theorem 3.11]{doi:10.1080/17442501003624407} only applies for square-integrable initial values and trace class Brownian motion. Latter is no restriction due to \cite[Proposition 4.7]{DePrato}. Because, due to Proposition \ref{prop: A3} in Appendix \ref{app: 1} and a Gronwall-type argument, weak existence and pathwise uniqueness holds for all  initial laws, a Yamada-Watanabe argument shows that also strong existence holds for arbitrary initial laws.} that for each \(m \in \mathbb{N}\) there exists a \cadlag adapted \(\mathbb{K}\)-valued process \(Y^m\) with dynamics
	\[
	\dd Y^m_t = b^{n, m} (Y^m_{t-}) \dd t + \sigma^{n, m} (Y^m_{t-})\dd W_t + \int_E v^{n, m}(y, Y^m_{t-})(p - q)(\dd y, \dd t), \quad Y^m_0 = \zeta.
	\]
	Define
	\[
	\tau_m \triangleq \inf (t \geq 0 \colon \|X_{t - }\kn \geq m \text{ or } \|X_t\kn \geq m).
	\] 
	We note that \(b^n(x) = b^{n, m} (x)= b^{n, m + 1} (x), \sigma^n(x) = \sigma^{n, m} (x)= \sigma^{n, m + 1}(x)\) and \(v^n(y, x) = v^{n, m}(y, x) = v^{n,m + 1}(y, x)\) for all \(x \in \mathbb{K} \colon \|x\kn \leq m\) and \(y \in E\). Thus, arguing as in the proof of \cite[Lemma 34.8]{metivier}, the Lipschitz conditions imply that a.s. \(Y^m_t = Y^{m + 1}_t\) for \(t \leq \tau_m \circ Y^m \wedge \tau_m \circ Y^{m + 1}\). Due to Galmarino's test (see \cite[Lemma III.2.43]{JS}), this yields that a.s.
	\(
	\tau_m \circ Y^m = \tau_m \circ Y^{m + 1}.
	\)
	Lemma \ref{lem: 2} implies that \(b^{n, m}, \sigma^{n, m}\) and \(v^{n, m}\) satisfy a linear growth condition with a linear growth constant independent of \(m\).
	Hence, an argument based on Doob's inequality and Gronwall's lemma (see the proof of Lemma \ref{lem: bounded exp} below) yields that a.s. \(\tau_m \circ Y^m \to \infty\) as \(m \to \infty\).
	Finally, we conclude that the process
	\[
	Y_t \triangleq Y^1_t \1 \{t < \tau_1 \circ Y^1\} + \sum_{k = 2}^\infty Y^k_t \1 \{\tau_{k - 1} \circ Y^{k-1} \leq t < \tau_k \circ Y^k\}, \quad t \geq 0, 
	\]
	solves the SDE \eqref{eq: SDE loc L}.
	Now, \cite[Theorem 3.13]{criens18} yields the claim.
	
	We stress that \(\bar{b}^n\) is well-defined due to Lemma \ref{lem: 2}.
	To see this, recall that there exists an \(\epsilon > 0\) such that \(h(x) =  x\) for all \(x \in \mathbb{B} \colon \|x\bn \leq \epsilon\). 
	Note that for \(x \in \mathbb{K}\) with \(\|x\bn \geq \epsilon\) we have \(\|x\kn \geq \|\iota\|_o^{-1} \|x\bn \geq \|\iota\|^{-1}_o \epsilon \triangleq \delta\).	
	Thus, for all \(x \in \mathbb{K}, y \in E\) we have 
	\begin{align*}
	\|h(v(x, y)) - v(x, y)\kn &= \| h(v(x, y)) - v(x, y)\kn \1_{\{\|v(x, y)\kn \geq \delta\}}
	\\&\leq \big(\tfrac{\|h\|_\infty}{\delta} + 1 \big) \|v(x, y)\kn \1_{\{\|v(x, y)\kn \geq \delta\}}
	\\&\leq \tfrac{1}{\delta} \big(\tfrac{\|h\|_\infty}{\delta} + 1 \big) \|v(x, y)\kn^2.
	\end{align*}
	Let \(l > 0\) be as in Lemma \ref{lem: 2} and set \(l' \triangleq \frac{l}{\delta} (\tfrac{\|h\|_\infty}{\delta} + 1 )\). We obtain for all bounded \(G \subset \mathbb{K}\) that
	\[
	\sup_{x \in G} \|\bar{b}^n(x)\kn \leq \sqrt{l} \sup_{x \in G} \big(1 + \|x\kn\big) + l' \sup_{x \in G} \big(1 + \|x\kn^2\big) \int_E\gamma^2(y)F(\dd y) < \infty, 
	\]
	which shows that \(\bar{b}^n\) is well-defined.
\end{proof}
Define by \(P^n (\dd \omega) \triangleq \bar{P}^n \big( (\iota X_t)_{t \geq 0} \in \dd \omega\big)\) a probability measure on \((\Omega, \mathcal{F})\), where \(X\) denotes the coordinate process on \(\Omega^o\).
\begin{lemma}
For each \(n \in \mathbb{N}\) the probability measure \(P^n\) solves the MP \((0, \tilde{b}^{n}, \tilde{a}^n, \tilde{K}^n, \eta \circ \iota^{-1})\), where for all \(x \in \mathbb{B}\)
\begin{align*}
\tilde{b}^{n} (x) &\triangleq b^n(x) + \int \big( h(v^n(y, x)) - v^n(y, x)\big)F(\dd y),\\ \tilde{a}^n (x)&\triangleq \sigma^n (x) \sigma^n (x)^*, \\ \tilde{K}^n(x, G) &\triangleq \int \1_{G\backslash \{0\}} (v^n(y, x)) F(\dd y),\quad G \in \mathcal{B}(\mathbb{B}).\end{align*}
Here, \(\sigma^n(x)^* \in L(\mathbb{B}, \mathbb{H})\) denotes the adjoint of \(\sigma^n(x)\in L (\mathbb{H}, \mathbb{B})\).
\end{lemma}
\begin{proof}
For \(x \in \mathbb{K}\) denotes by \(\sigma^n (x)^\star \in L(\mathbb{K}, \mathbb{H})\) the adjoint of \(\sigma^n(x) \in L(\mathbb{H}, \mathbb{K})\) and by \(\sigma^n(x)^* \in L(\mathbb{B}, \mathbb{H})\) the adjoint of \(\sigma^n(x) \in  L(\mathbb{H}, \mathbb{B})\). 
Recalling hypothesis (i), for all \(x \in \mathbb{K}\) and \(y, z \in \mathbb{B}\) we obtain that
\begin{align*}
\langle \sigma^n(\iota x) \sigma^n(\iota x)^* z, y\rangle_\mathbb{B} 
& = \langle \sigma^n(x) \sigma^n(x)^* z, \iota^* y\rangle_\mathbb{K} 
\\&= \langle \sigma^n(x)^* z, \sigma^n(x)^\star i^* z \rangle_\mathbb{H} \\&= \langle z, \sigma^n(x) \sigma^n(x)^\star i^* y\rangle_\mathbb{B}
\\&= \langle \iota^* z, \sigma^n(x) \sigma^n(x)^\star i^* y\rangle_\mathbb{K}
\\&= \langle \sigma^n(x) \sigma^n(x)^\star i^* z, \iota^* y\rangle_\mathbb{K}.
\end{align*}
Denote by \(X\) the coordinate process on \(\Omega^o\), take \(f = g(\langle \cdot, y_1\rangle_\mathbb{B}, \dots, \langle \cdot, y_n\rangle_\mathbb{B}) \in \mathcal{C}\)
and define \(f^* \triangleq g(\langle \cdot, \iota^* y_1\rangle_\mathbb{K}, \dots, \langle \cdot, \iota^* y_n\rangle_\mathbb{K})\). We note that
\begin{align*}
\sum_{i = 1}^n &\langle b(\iota X), y_i\rangle_\mathbb{B} \partial_i f (\iota X) + \frac{1}{2} \sum_{i = 1}^n\sum_{j = 1}^n \langle \sigma^n (\iota X) \sigma^n(\iota X)^* y_i,  y_j\rangle_\mathbb{B} \partial_{ij}^2 f (\iota X) \\&\quad\qquad\qquad\qquad + \int_E \Big(f (\iota X + v(\iota X, z)) - f (\iota X) - \sum_{i = 1}^n \langle v(\iota X, z), y_i \rangle_\mathbb{B} \partial_i f(\iota X)\Big) F(\dd z)
\\&= \sum_{i = 1}^n \langle b(X), \iota^* y_i\rangle_\mathbb{K} \partial_i f^* (X) + \frac{1}{2} \sum_{i = 1}^n\sum_{j = 1}^n \langle \sigma^n (X) \sigma^n(X)^\star \iota^* y_i, \iota^* y_j\rangle_\mathbb{K} \partial_{ij}^2 f^* (X) \\&\quad\qquad\qquad\qquad+ \int_E \Big(f^* (X + v(X, z)) - f^* (X) - \sum_{i = 1}^n \langle v(X, z), \iota^* y_i \rangle_\mathbb{K} \partial_i f^*(X)\Big) F(\dd z).
\end{align*}
Now, the claim follows from the definition of the martingale problem and Lemma \ref{lem: sde}. We stress that a similar argument as used for \(\bar{b}^n\) in the proof of Lemma \ref{lem: sde} shows that \(\tilde{b}^n\) is well-defined, see Lemma \ref{eq: loc bdd}. 
\end{proof}

\begin{lemma}\label{lem: 4}
	The sequence \((P^n)_{n \in \mathbb{N}}\) is tight.
\end{lemma}
\begin{proof}
	We start with a moment bound. Fix \(\epsilon > 0\). Because any Borel probability measure on a Polish space is tight, we find a compact set \(K \subset \mathbb{K}\) such that \(\eta (K) \geq 1-  \frac{\epsilon}{2}\).
	Define \(Z \triangleq \1 \{X_0 \in K\}\), which is a random variable on \((\Omega^o, \mathcal{F}^o)\).
	\begin{lemma}\label{lem: bounded exp}
		For all \(T \in \mathbb{R}_+\) we have
			\(\sup_{n \in \mathbb{N}}  \bar{E}^n \big[\sup_{t \in [0, T]}\|X_{t}\|^2_\mathbb{K} Z \big] < \infty\). 
	\end{lemma}
	\begin{proof}
		For \(m \in \mathbb{N}\) let \(\tau_m\) be defined as in \eqref{eq: tau}. Fix \(t \in [0, T]\).
		Using Doob's and H\"older's inequality, \cite[Lemma 4.7]{criens18} and \cite[Proposition II.1.28, Theorems II.1.33, II.2.34]{JS} we obtain that 
		\begin{align*}
		\bar{E}^n \Big[ \sup_{s \in [0, t]} \langle X_{s \wedge \tau_m}, e_k\ks^2 Z \Big]
		\leq 16\ \Big( \bar{E}^n \Big[ Z &\int_0^{t \wedge \tau_m} \langle \bar{a}^n(X_s) e_k, e_k\ks \dd s + TZ \int_0^{t \wedge \tau_m} \langle b^n(X_s), e_k\ks^2 \dd s \Big]   \\&+ \bar{E}^n\Big[ Z\int_0^{t \wedge \tau_m} \int_{\mathbb{K}} \langle x, e_k\ks^2 \bar{K}^n(X_s, \dd x) \dd s  + Z\langle X_0, e_k\ks^2\Big]\Big).
		\end{align*}
		Using the Parseval identity, we obtain 
		\begin{align*}
		\bar{E}^n \Big[ \sup_{s \in [0, t]}\| X_{s \wedge \tau_m}\kn^2 Z \Big] &\leq \sum_{k = 1}^\infty \bar{E}^n \Big[ \sup_{s \in [0, t]} \langle X_{s \wedge \tau_m}, e_k\ks^2 Z \Big] \\&
	\leq 16\ \Big( \bar{E}^n \Big[ Z \int_0^{t \wedge \tau_m} \textup{trace}_\mathbb{K} \bar{a}^n(X_s) \dd s + TZ \int_0^{t \wedge \tau_m} \| b^n(X_s)\kn^2 \dd s \Big]   \\&\hspace{2.5cm}+ \bar{E}^n\Big[ Z\int_0^{t \wedge \tau_m} \int_{\mathbb{K}} \| x\kn^2 \bar{K}^n(X_s, \dd x) \dd s  + Z\|X_0\kn^2\Big]\Big).
		\end{align*}
Due to Lemma \ref{lem: 2}, we find two constants \(c_1, c_2 > 0\) independent of \(n, m\) and \(t\), such that
		\[
		\bar{E}^n \Big[ \sup_{s \in [0, t]} \|X_{s \wedge \tau_{m}}\kn^2 Z \Big] \leq c_1 + c_2 \int_0^t \bar{E}^n \Big[ \sup_{s \in [0, r]} \|X_{s \wedge \tau_{m}}\kn^2 Z\Big] \dd r.
		\]
		Gronwall's lemma yields that 
		\[
		\bar{E}^n \Big[ \sup_{t \in [0, T]}\|X_{t \wedge \tau_{m}}\kn^2 Z \Big] \leq c_1 e^{c_2 T}.
		\]
		Finally, Fatou's lemma yields that 
		\[
		\bar{E}^n \Big[ \sup_{s \in [0, T]} \|X_{t}\kn^2 Z \Big] \leq \liminf_{m \to \infty} \bar{E}^n \Big[ \sup_{t \in [0, T]} \|X_{t \wedge \tau_{m}}\kn^2 Z \Big] \leq c_1 e^{c_2 T}.
		\]
		This completes the proof.
	\end{proof}
Fix \(t \in \mathbb{R}_+\).
	Using Chebyshev's inequality, we obtain that 
	\[
	\bar{P}^n \big(\|X_t\kn Z \leq R \big) \geq 1 - \frac{\sup_{n \in \mathbb{N}} \bar{E}^n \big[\|X_t\kn^2 Z\big]}{R^2}.
	\]
	Due to Lemma \ref{lem: bounded exp}, we find \(R^* > 0\) such that 
	\[
	\inf_{n \in \mathbb{N}} \bar{P}^n \big(\|X_t\kn Z \leq R^* \big) \geq 1 - \tfrac{\epsilon}{2}.
	\]
	Set 
	\[
	K_1 \triangleq \textup{cl}_\mathbb{B} \big\{\iota x \colon x \in \mathbb{K} \text{ and } \|x\kn \leq R^* \big\} \subset \mathbb{B}.
	\]
	Due to Lemma \ref{lem: emb comp}, the embedding \(\iota\) is compact and, consequently, the set \(K_1\) is compact in \(\mathbb{B}\). We obtain 
	\begin{align*}
	P^n \big(X_t \in K_1\big) &\geq \bar{P}^n \big(\iota X_t \in K_1, X_0 \in K \big) 
	\\&\geq \bar{P}^n \big(\| X_t \kn Z \leq R^*\big) - 1 + \iota(K)
	\\&\geq \bar{P}^n \big(\|X_t\kn Z \leq R^* \big) - \tfrac{\epsilon}{2}
	\\&\geq 1 - \epsilon.
	\end{align*}
	This shows that \((P^n \circ X_t^{-1})_{n \in \mathbb{N}}\) is tight for all \(t \in \mathbb{R}_+\). In other words, (i) in Proposition \ref{prop: gen tight} holds.
	
	Let \(M > 0\) and \((\rho_n)_{n \in \mathbb{N}}\) be a sequence of stopping times on \(\Omega\) such that \(\sup_{n \in \mathbb{N}} \rho_n \leq M\). Moreover, let \((h_n)_{n \in \mathbb{N}} \subset \mathbb{R}_+\) be a sequence such that \(h_n \to 0\) as \(n \to \infty\).
	We can consider \((\rho_n)_{n \in \mathbb{N}}\) as a sequence of stopping times on \(\Omega^o\) by identifying \(\rho_n\) with \(\rho_n \circ \iota\).

	Let \(\epsilon > 0, K \subset \mathbb{K}\) and \(Z\) be as above.
	We can argue as in the proof of Lemma \ref{lem: bounded exp} and obtain that
	\begin{align*}
	\bar{E}^n \big[ &\|X_{\tau_m \wedge (\rho_n + h_n)} - X_{\tau_m \wedge \rho_n}\kn^2 Z \big] \\&\leq 12\ \Big( \bar{E}^n \Big[ Z\int_{\tau_m \wedge \rho_n}^{\tau_m \wedge (\rho_n + h_n)} \textup{ trace}_\mathbb{K} \bar{a}^n(X_s) \dd s  + Z\int_{\tau_m \wedge \rho_n}^{\tau_m \wedge (\rho_n + h_n)} \int_{\mathbb{K}} \| x \kn^2 \bar{K}^n(X_s, \dd x) \dd s  \Big] \\&\qquad\qquad + \bar{E}^n \Big[ h_nZ \int_{\tau_m \wedge \rho_n}^{\tau_m \wedge (\rho_n + h_n)} \| b^n(X_s)\kn^2 \dd s \Big]\Big).
	\end{align*}
	Thus, due to Lemmata \ref{lem: 2} and \ref{lem: bounded exp}, we find two constants \(c_1, c_2 > 0\) independent of \(n\) and \(m\) such that 
	\[
	\bar{E}^n \big[ \|X_{\tau_m \wedge (\rho_n + h_n)} - X_{\tau_m \wedge \rho_n}\kn^2Z \big] \leq c_1 h_n + c_2 h^2_n.
	\]
	Furthermore, Fatou's lemma yields that 
	\[
	\bar{E}^n \big[ \|X_{\rho_n + h_n} - X_{\rho_n}\kn^2Z \big] \leq c_1 h_n + c_2 h^2_n.
	\]
	Fix \(\delta > 0\) and set \(\bar{\delta} \triangleq \delta \|\iota\|_o^{-1}\). By Chebyshev's inequality, we have  
	\begin{align*}
	\bar{P}^n \big(\|X_{\rho_n + h_n} - X_{\rho_n}\kn \geq \bar{\delta}\big) &\leq \bar{P}^n \big(\|X_{\rho_n + h_n} - X_{\rho_n}\kn Z \geq \bar{\delta}\big) + \eta(K^c)
	\\&\leq \frac{\bar{E}^n \big[ \|X_{\rho_n + h_n} - X_{\rho_n}\kn^2Z \big]}{\bar{\delta}^2} + \frac{\epsilon}{2}
	\\&\leq \frac{c_1 h_n + c_2 h^2_n}{\bar{\delta}^2} + \frac{\epsilon}{2}.
	\end{align*}
	Because \(h_n \to 0\) as \(n \to \infty\), we find an \(N \in \mathbb{N}\) such that 
	\[
	\bar{P}^n \big(\|X_{\rho_n + h_n} - X_{\rho_n}\kn \geq \bar{\delta}\big) < \epsilon
	\]
	for all \(n \geq N\).
	We conclude that
	\[
	 P^n \big(\|X_{\rho_n + h_n} - X_{\rho_n}\bn \geq \delta\big) \leq \bar{P}^n \big(\|X_{\rho_n + h_n} - X_{ \rho_n}\kn \geq \bar{\delta}\big) < \epsilon
	\]
	for all \(n \geq N\),
	i.e. (ii) in Proposition \ref{prop: gen tight} holds. This completes the proof.
\end{proof}
Define \(\mathcal{K}\) as in \eqref{K} with \(A = 0\) and \(b\) replaced by \(\mu\) as given in the statement of Theorem \ref{prop: existence} and let \(\mathcal{K}^n\) as in \eqref{K} with \(A = 0\) and \(b\) replaced by \(\tilde{b}^n\), \(a\) replaced by \(\tilde{a}^n\) and \(K\) replaced by \(\tilde{K}^n\).

For all \(f = g(\langle \cdot, y^*\bs) \in \mathcal{D}\), Taylor's theorem yields that 
\[
\big|f(x + v(x, y)) - f(x) - \partial f (x) \langle v(x, y), y^*\bs \big| \leq \tfrac{1}{2} \|g''\|_\infty \|y^*\|^2_\mathbb{B} \|v(x, y)\|^2_\mathbb{B}, \qquad y \in E, x \in \mathbb{B}.
\]
Thus, due to the continuity assumptions on \(b, \sigma\) and \(v\), \eqref{eq: loc bdd B} and the dominated convergence theorem the map
\[
\mathbb{B} \ni x \mapsto \mathcal{K} f(x)
\]
is continuous for all \(f \in \mathcal{D}\). 
Let \(K \subset \mathbb{B}\) be compact and \(f = g(\langle \cdot, y^*\bs) \in \mathcal{D}\).
Note that for all \(x \in \mathbb{B}\) and \(y \in E\)
\begin{align*}
K^n (x, y) \triangleq \big|f(x + v(x, y)) &- f(x + v^n(x, y)) - \partial f (x) \langle v(x,y) - v^n(x, y), y^*\bs \big|
\\&\leq \|g'\|_\infty \big|\langle v(x, y) - v^n(x, y), y^*\bs\big| + \|g'\|_\infty \big| \langle v(x, y) - v^n(x, y), y^*\bs\big|
\\&\leq 2 \|g'\|_\infty \|y^*\bn \|v(x, y) - v^n(x, y)\bn.
\end{align*}
Thus, Lemma \ref{lem: 3} yields that for all \(y \in E\) 
\begin{align}\label{eq: pw conv}
\sup_{x \in K}  K^n(x, y)  \to 0 \text{ as } n \to \infty.
\end{align}
Due to Taylor's theorem, we have for all \(x \in \mathbb{B}\) and \(y \in E\)
\begin{align*}
K^n(x, y)  &\leq \tfrac{1}{2} \|g''\|_\infty \|y^*\bn^2 \big( \|v (x, y) \bn^2 + \|v^n(x, y)\bn^2\big).
\end{align*}
Consequently, recalling Lemma \ref{lem: 3}, hypothesis (iii) and \eqref{eq: pw conv},  we deduce from the dominated convergence theorem that 
\[
\int_E\sup_{x \in K} K^n(x, y) F(\dd y) \to 0 \text{ as } n \to \infty.
\]
Using again Lemma \ref{lem: 3}, we obtain that
\begin{equation*}\begin{split}
\sup_{x \in K}\big| \mathcal{K} f (x) - \mathcal{K}^n f(x)\big|  &\leq \|g'\|_\infty \|y^*\bn \sup_{x \in K} \|b^n (x) - b(x)\bn \\&\qquad + \tfrac{1}{2} \|g''\|_\infty \|y^*\bn^2 \sup_{x \in K} \|\sigma^n (x) - \sigma (x)\|_{L(\mathbb{B})}^2 \\&\qquad+ \int_E \sup_{x \in K} K^n(x,y) F(\dd y) \\&\to 0 \text{ as } n \to \infty.
\end{split}
\end{equation*}
Using again Taylor's theorem, Lemma \ref{eq: loc bdd} yields that 
\begin{align*}
\sup_{n \in \mathbb{N}} \sup_{\|x\bn \leq m} \big| \mathcal{K}^n f (x) \big| &\leq \|g'\|_\infty \|y^* \bn \sup_{n \in \mathbb{N}} \sup_{\|x\bn \leq m} \|b^n(x)\bn 
\\&\qquad+ \tfrac{1}{2} \|g''\|_\infty \|y^*\bn^2 \sup_{n \in \mathbb{N}} \sup_{\|x\bn \leq m} \|\sigma^n(x)\|_{L(\mathbb{B})}^2 
\\&\qquad + \tfrac{1}{2} \|g''\|_\infty \|y^*\bn^2 \sup_{n \in \mathbb{N}} \sup_{\|x\bn \leq m} \int_E \|v^n(x, y)\bn^2 F(\dd y) < \infty.
\end{align*}
 We conclude from Corollary \ref{coro: existence} and Lemma \ref{lem: 4} that the MP \((0, \mu, a, K, \eta \circ \iota^{-1})\) has a solution.
\qed

\subsection{Proof of Proposition \ref{prop: tight}}\label{proof: prop1}
For \(T \in (0, \infty)\) let \(\Omega^T\) be the space of continuous functions \([0, T] \to \mathbb{B}\), let \(X^T\) be the coordinate process on \(\Omega^T\) and set \(\mathcal{F}^T \triangleq \sigma (X^T_t, t \in [0, T])\). 

Due to \cite[Corollary 5]{whitt1970}, the sequence \((P^n \circ X_{\cdot \wedge \tau_m}^{-1})_{n \in \mathbb{N}}\) is tight if and only if for all \(T \in \mathbb{N}\) its restriction to \((\Omega^T, \mathcal{F}^T)\) is tight.

We fix \(T \in \mathbb{N}\).
Our strategy is to adapt the compactness method from \cite{doi:10.1080/17442509408833868}. 
In the following we fix \(n \in \mathbb{N}\) and work with the filtered probability space \((\Omega, \mathscr{F}, \F, P^n)\).
Due to \cite[Theorem 3.6]{EJP2924}, we find a cylindrical standard Brownian motion \(W\) (possibly on an extension of \((\Omega, \mathscr{F}, \F, P^n)\)) such that 
\[
X_t = S_t X_0 + \int_0^t S_{t - s} b^n(X_s)\dd s + \int_0^t S_{t-s} \sigma^n(X_s) \dd W_s, \quad t \in [0, T].
\]
Let \(\varepsilon > 0\) be fixed. 
Because  \((\eta^n)_{n  \in  \mathbb{N}}\) is tight, there exists a compact set \(K \subset \mathbb{B}\) such that 
\[
\inf_{n \in \mathbb{N}} \eta^n (K) \geq 1 - \tfrac{\varepsilon}{2}.
\]
Let \(m^* > 0\) be a constant such that \(\|x\| \leq m^*\) for all \(x \in K\).

In the cases (ii) and (iv), let \(\theta \triangleq \lambda\) and choose \(p > 2\) such that \(\tfrac{1}{p} < \theta\). In the other cases (i) and (iii) take \(p > 2\) and choose \(\theta \in (0, \frac{1}{2})\) such that \(\tfrac{1}{p} < \theta\).

In the cases (iii) and (iv), the linear growth conditions and a standard argument based on Gronwall's lemma (see \cite[Theorems 7.2 and 7.5]{DePrato}) yields that
\[
\sup_{s \in [0, T]} E^{P^n}\big[ \1_K (X_0) \|X_{s \wedge \tau_\infty}\|^p\big] = \sup_{s \in [0, T]} E^{P^n}\big[ \1_K(X_0) \|X_s\|^p\big] \leq C_{p, T}, 
\]
where \(C_{p, T} \in (0, \infty)\) is a constant independent of \(n\). 
Note that 
\begin{equation}\label{eq: b ineq}
\begin{split}
E^{P^n} \Big[ \1_K(X_0) &\int_0^T \|b^{n}(X_{s \wedge \tau_m})\|^p \dd s \Big] \\&\leq \begin{cases}
T\cdot\big(\sup_{k \in \mathbb{N}} \sup_{\|x\| \leq m \vee m^*} \|b^k(x)\|\big)^p,&\text{if (i) or (ii) holds},\\
TK^p 2^{p} \big(1 + C_{p, T}\big),&\text{if (iii) or (iv) holds}.
\end{cases}
\end{split}
\end{equation}

Let \(L^p([0, T], \mathbb{B})\) be the space of all Borel functions \(f \colon [0, T]\to \mathbb{B}\) such that \(\int_0^T \|f(s)\|^p \dd s < \infty\).
For \(f \in L^p([0, T], \mathbb{B})\), we define  
\[
(G_{\xi} f)(t)  \triangleq \int_0^{t} (t - s)^{\xi - 1} S_{t - s} f(s) \dd s,\quad \xi \in (\tfrac{1}{p}, 1], t \in [0, T].
\]
Due to \cite[Proposition 8.4]{DePrato}, the operator \(G_\xi\) is compact from \(L^p ([0, T], \mathbb{B})\) into \(\Omega^T\).
We define 
\[
Y^{n, m, \theta}_s \triangleq \int_0^s (s - r)^{- \theta} S_{s-r} \sigma^{n}(X_{r \wedge \tau_m})\dd W_r,\quad s \in [0,T].
\]

In either of the cases (i) and (iii), we obtain as in the proof of \cite[Proposition 7.3]{DePrato} that there exists a constant \(\widehat{C}_{p, T} \in (0, \infty)\) such that
\begin{equation}\label{bound sigma}
\begin{split}
E^{P^n}\Big[ \1_K(X_0) &\int_0^T \|Y^{n, m, \theta}_s\|^p \dd s\Big] \\&\leq \widehat{C}_{p, T} \int_0^T E^{P^n} \Big[ \1_K(X_0) \| \sigma^{n} (X_{s \wedge \tau_m}) \|^p_\textup{HS}\Big] \dd s 
\\&\leq \begin{cases}
\widehat{C}_{p, T} T \big(\sup_{k \in \mathbb{N}}\sup_{\|x\| \leq m \vee m^*} \|\sigma^k(x)\|_\text{HS}\big)^p,&\text{if (i) holds},\\
\widehat{C}_{p, T} T K^p 2^{p} \big(1 + C_{p, T} \big),&\text{if (iii) holds}.
\end{cases}
\end{split}
\end{equation}

In the cases (ii) and (iv) we have chosen \(\theta = \lambda\) and we obtain as in the proof of \cite[Proposition 6.3.5]{roeckner15} that there exists a constant \(\widehat{C}_{p} \in (0, \infty)\) such that
\begin{equation}\label{bound sigma2}
\begin{split}
E^{P^n}\Big[ \1_K&(X_0) \int_0^T \|Y^{n, m, \theta}_s\|^p \dd s\Big] \\&\leq \widehat{C}_{p} \Big(\int_0^T (T - s)^{- 2 \lambda} \|S_{T - s}\|^2_\textup{HS} E^{P^n} \Big[ \1_K(X_0) \| \sigma^{n} (X_{s \wedge \tau_m}) \|^p_o\Big]^\frac{2}{p} \dd s \Big)^\frac{p}{2}
\\&\leq \begin{cases}
\widehat{C}_{p} \big(\int_0^T s^{- 2 \lambda}\|S_s\|^2_\textup{HS} \dd s\big)^\frac{p}{2} \big(\sup_{k \in \mathbb{N}}\sup_{\|x\| \leq m \vee m^*} \|\sigma^k(x)\|_o\big)^p,&\text{if (ii) holds},\\
\widehat{C}_{p}\big( \int_0^T s^{- 2 \lambda}\|S_s\|^2_\textup{HS} \dd s \big)^{\frac{p}{2}}K^p 2^{p} \big(1 + C_{p, T} \big),&\text{if (iv) holds}.
\end{cases}
\end{split}
\end{equation}
Recall the the r.h.s. is finite due to Lemma \ref{lem: for all conv}.

By \cite[Proposition 1]{doi:10.1080/17442508708833480}, we have the following factorization formula
\[
\int_0^t S_{t-s} \sigma^{n}(X_{s \wedge \tau_m}) \dd W_s = \tfrac{\sin (\pi \theta)}{\pi} \big(G_\theta Y^{n, m, \theta}\big) (t),\quad t \in [0, T].
\]
Thus, due to \cite[Remark A.1]{Brzeniak2005}, for all \(t \in [0, T]\)
\[
\Big(\int_0^\cdot S_{\cdot -s} \sigma^n(X_s) \dd W_s \Big) (t \wedge \tau_m) = \tfrac{\sin (\pi \theta)}{\pi} \big(G_\theta Y^{n, m, \theta}\big) (t \wedge \tau_m).
\]
Consequently, we obtain
\begin{align*}
X_{t \wedge \tau_m} 
= S_{t \wedge \tau_m} X_0 + \left(G_1 b^{n} (X_{\cdot \wedge \tau_m})\right)(t \wedge \tau_m)  + \tfrac{\sin(\pi \theta)}{\pi} \big(G_\theta Y^{n, m, \theta}\big)(t \wedge \tau_m),\quad t \in [0, T].
\end{align*}
Define \begin{align*}B_R &\triangleq \Big\{u \in L^p ([0, T], \mathbb{B}) \colon \int_0^T \|u(s)\|^p \dd s \leq R \Big\},\\
\Lambda (R, \theta, m) &\triangleq \left\{ \alpha \in \Omega^T \colon \alpha = (G_\theta u)(\cdot \wedge \tau_m(\lambda)), u \in B_R, \lambda \in \Omega\right\},
\\
\Lambda^\star (m) &\triangleq \left\{\alpha \in \Omega^T \colon \alpha = S_{\cdot \wedge \tau_m(\lambda)} x, x \in K, \lambda \in \Omega\right\},
\end{align*}
and 
\begin{align*}
\Lambda' &(R, m) 
\\&\triangleq \left\{ \alpha \in \Omega^T \colon \alpha = \alpha^1 + \alpha^2 + \tfrac{\sin(\pi \theta)}{\pi} \alpha^3,\  \alpha^1 \in \Lambda^\star (m), \alpha^2 \in \Lambda (R, 1, m), \alpha^3 \in \Lambda (R, \theta, m) \right\}.
\end{align*}
We claim that the sets \(\Lambda^\star (m), \Lambda (R, 1, m)\) and \(\Lambda (R, \theta, m)\) are relatively compact. 
To show this, we recall the following version of the Arzel\`a-Ascoli theorem (see \cite[Theorem A.2.1]{Kallenberg} and \cite[Theorem 1.8.23]{tao2010epsilon}): 
\begin{fact}\label{fact2} Fix two metric spaces \(U\) and \((V, \rho)\), where \(U\) is compact and \(V\) is complete. Let \(D\) be dense in \(U\). A set \(A \subset C(U, V)\) (the space of all continuous functions \(U \to V\) equipped with the uniform topology) is relatively compact if and only if for all \(t \in D\) the set \(\pi_t (A) \triangleq \{\alpha(t)\colon \alpha \in A\} \subset V\) is relatively compact in \(V\) and \(A\) is equicontinuous, i.e. for all \(s \in U\) and \(\epsilon > 0\) there exists a neighborhood \(N\subset U\) of \(s\) such that for all \(r \in N\) and \(\alpha \in A\)
	\[
	\rho (\alpha(r), \alpha(s)) < \epsilon.
	\]
	In that case, even \(\bigcup_{t \in U} \pi_t(A)\) is relatively compact in \(V\).
\end{fact}
Let \(\xi \in \{\theta, 1\}\).
The set \(\Lambda (R, \xi, m)\) is relatively compact if it is equicontinuous and for all \(t \in [0, T]\) the set
\[
C_t \triangleq \{y \in \mathbb{B}\colon y = \alpha(t), \alpha \in \Lambda(R, \xi, m)\}
\]
is relatively compact in \(\mathbb{B}\).
We note that the set \(\Lambda (R, \xi, \infty)\) is relatively compact, because the operator \(G_\xi\) is compact. 
Thus, Fact \ref{fact2} yields that the set \[C \triangleq \bigcup_{t \in [0, T]}\{y \in \mathbb{B} \colon y = \alpha(t), \alpha \in \Lambda (R, \xi, \infty)\}\] is relatively compact in \(\mathbb{B}\). For all \(t \in [0, T]\) we have \(C_t \subset C\) and hence also the set \(C_t\) is relatively compact in \(\mathbb{B}\). 
Let now \(t \in [0, T]\) and \(\epsilon > 0\). 
Again due to Fact \ref{fact2}, there exists a \(\delta> 0\) such that for all \(s \in [0, T]\) with \(|t - s| \leq \delta\) we have 
\begin{align*}
\|\alpha (t) - \alpha (s)\| < \tfrac{\epsilon}{2}
\end{align*}
for all \(\alpha \in \Lambda (R, \xi, \infty)\). 
Note that 
\[
\|\alpha(t \wedge \tau_m(\lambda)) - \alpha (s \wedge \tau_m(\lambda))\| = \begin{cases} \|\alpha(\tau_m(\lambda)) - \alpha(s)\|,& t \geq \tau_m(\lambda)\geq s,\\
\|\alpha(t) - \alpha(\tau_m(\lambda))\|,&t \leq \tau_m(\lambda) \leq s,\\
\|\alpha(t) - \alpha(s)\|,&t, s \leq \tau_m(\alpha),\\
0,&t, s \geq \tau_m(\alpha).\end{cases}
\]
In the first case, we have 
\[
\|\alpha(\tau_m(\lambda)) - \alpha(s)\| \leq \|\alpha(\tau_m(\lambda)) - \alpha(t)\| + \|\alpha(t) - \alpha(s)\| < \epsilon, 
\]
because \(|t - \tau_m(\lambda)| \leq |t - s| \leq \delta\). In the second case, we have 
\[|t - \tau_m(\lambda)| = \tau_m(\lambda) - t \leq s - t \leq \delta,\] and hence 
\[
\|\alpha(t) - \alpha(\tau_m(\lambda))\| < \tfrac{\epsilon}{2} \leq \epsilon. 
\]
In the third and fourth case, the desired inequality holds trivially. Thus, 
the inequality 
\[
\|\alpha(t) - \alpha(s)\| < \epsilon 
\] holds for all \(\alpha \in \Lambda (R, \xi, m)\) and \(s \in  [0, T]\) with \(|t - s| \leq \delta\). We conclude that the set \(\Lambda(R, \xi, m)\) is equicontinuous. Now, Fact \ref{fact2}  yields that it is also relatively compact.

The relative compactness of \(\Lambda^\star (m)\) follows from the same argument, if we can show that the set \(\Lambda^\star (\infty)\) is relatively compact.
Because \(K\) and \(S_t\) for all \(t > 0\) are compact, the set
\[
\{y \in \mathbb{B} \colon y = S_t x, x \in K\},\quad t \in (0, T],
\]
is relatively compact in \(\mathbb{B}\). Since \((S_t)_{t \geq 0}\) is a \(C_0\)-semigroup, \cite[Lemma 5.2, p. 37]{engel2006one} yields that the map
\[
[0, T] \times K \ni (t, x) \mapsto S_t x
\]
is uniformly continuous. 
In other words, for all \(\epsilon > 0\) there exists a \(\delta > 0\) such that 
\[
\|S_t x - S_s x\| < \epsilon
\]
for all \(x \in K\) and \(t, s \in [0, T] \colon |t - s| < \delta\), i.e. the set \(\Lambda^\star (\infty)\) is equicontinuous. Thus, it is relatively compact by Fact \ref{fact2}. We conclude that also \(\Lambda^\star (m)\) is relatively compact.

Finally, it follows that \(\Lambda'(R, m)\) is relatively compact. Using \eqref{eq: b ineq}, \eqref{bound sigma} and \eqref{bound sigma2} together with Chebyshev's inequality, for all \(R > 0\) there exists a constant \(C'\) independent of \(n\) and \(R\) such that
\begin{align*}
P^n &\Big(X_0 \in K\ \text{ and } \int_0^T \|b^n (X_{s \wedge \tau_m})\|^p \dd s \leq R \ \text{ and }\ \int_0^T\|Y^{n, m, \theta}_s\|^p \dd s \leq R\Big) 
\\&\geq \eta^n (K)- P^n \Big( \1_K (X_0) \int_0^T \|b^n (X_{s \wedge \tau_m})\|^p \dd s > R\Big) - P^n \Big(\1_K(X_0) \int_0^T\|Y^{n, m, \theta}_s\|^p \dd s > R\Big)
\\&\geq 1 - \frac{\varepsilon}{2} - \frac{1}{R}\ \Big( E^{P^n} \Big[ \1_K(X_0)\int_0^T \|b^n (X_{s \wedge \tau_m})\|^p \dd s\Big] + E^{P^n} \Big[  \1_K(X_0)\int_0^T\|Y^{n, m, \theta}_s\|^p \dd s\Big]\Big)
\\&\geq 1 - \frac{\varepsilon}{2} -\frac{C'}{R}.
\end{align*}
Consequently, there exists an \(R_\varepsilon > 0\) such that 
\begin{align*}
\inf_{n \in \mathbb{N}} P^n\big((X_{t \wedge \tau_m})_{t \in [0, T]} \in \Lambda'(R_\varepsilon, m)\big) 
\geq 1 - \varepsilon.
\end{align*}
We conclude that the restriction of \((P^n \circ X_{\cdot \wedge \tau_m}^{-1})_{n \in \mathbb{N}}\) to \((\Omega^T, \mathcal{F}^T)\) is tight. Because \(T \in \mathbb{N}\) was arbitrary, this implies that \((P^n \circ X_{\cdot \wedge \tau_m}^{-1})_{n \in \mathbb{N}}\) is tight.
\qed
\appendix
\section{Some Facts on Cylindrical Martingale Problems}\label{app: 1}
To keep this article self-continuous we collect some results on cylindrical martingale problems, which are used in the proofs of Theorems \ref{theo1} and \ref{prop: loc}.
Let \(M^f\) be defined as in \eqref{f - K}, \(\tau_z\) be defined as in \eqref{eq: tau} and recall that
		\begin{align*}
		\mathcal{D} &= \big\{ g(\langle \cdot, y^*\rangle) \colon g \in C^2_c(\mathbb{R}), y^* \in D(A^*)\big\},\\
		\mathcal{B} &= \big\{ g (\langle \cdot, y^*\rangle )\colon g \in C^2_b(\mathbb{R}), y^* \in D(A^*) \big\}.
		\end{align*}
		\begin{proposition}\label{prop: A2}
		The following are equivalent:
		\begin{enumerate}
			\item[\textup{(i)}]
			\(P \in \mathcal{M}(A, b, a, K, \eta)\).
			\item[\textup{(ii)}]
			\(P \circ X^{-1}_0 = \eta\) and for all \(n \in \mathbb{N}\) and \(f \in \mathcal{D}\) the process \(M^f_{\cdot \wedge \tau_n}\) is a \(P\)-martingale. 
		\end{enumerate}
\end{proposition}
\begin{proof}
	See \cite[Lemma 4.7]{criens18}.
\end{proof}
\begin{proposition}\label{prop: A1}
	\begin{enumerate}
		\item[\textup{(i)}] For all \(f \in \mathcal{B}, t \in \mathbb{R}_+, n \in \mathbb{N}\) and all bounded \(G \subset \mathbb{B}\) it holds that 
		\[
		\sup_{x \in G}  \big| \mathcal{K} f (x) \big| + \sup_{s \in [0, t]} \sup_{\omega \in \Omega} \big| M^f_{s \wedge \tau_n(\omega)}(\omega)\big| < \infty.
		\]
		\item[\textup{(ii)}]
		If \(P\in\mathcal{M}(A, b, a, K, \eta)\), 
		then for all \(n \in \mathbb{N}\) and \(f \in \mathcal{B}\) the process \(M^f_{\cdot \wedge \tau_n}\)  is a \(P\)-martingale. 
	\end{enumerate}
\end{proposition}
\begin{proof}
	For (i) see \cite[Lemma 4.5]{criens18} including its proof and for (ii) see \cite[Corollary 4.6]{criens18}.
\end{proof}
\begin{proposition}\label{prop: A3}
	Suppose that for all \(x \in \mathbb{B}\) the MP \((A, b, a, K, \varepsilon_x)\) has a unique solution \(P_x\). 
	\begin{enumerate}
		\item[\textup{(i)}] 
		The map \(x \mapsto P_x\) is Borel and for all Borel probability measures \(\eta\) on \(\mathbb{B}\) the MP \((A, b, a, K, \eta)\) has a unique solution given by \(\int P_x \eta(\dd x)\).
		\item[\textup{(ii)}] Let \(\rho\) be a stopping time and \(P\) be a probability measure on \((\Omega, \mathcal{F})\) such that \(P \circ X^{-1}_ 0 = \eta\) and for all \(f \in \mathcal{D}\) the process \(M^f_{\cdot \wedge \rho}\) is a local \(P\)-martingale. Then, \(P = \int P_x \eta(\dd x)\) on \(\mathcal{F}_\rho\).
	\end{enumerate}
\end{proposition}
\begin{proof} 
	For (i) see \cite[Theorem 3.2]{criens18} and for (ii) see \cite[Proposition 4.13]{criens18}.
\end{proof}

\bibliographystyle{acm}


\end{document}